\documentclass[letterpaper,11pt]{article}

\usepackage[utf8]{inputenc}
\usepackage[english]{babel}

\usepackage{latexsym}
\usepackage{amssymb}
\usepackage{amsthm}
\usepackage{mathrsfs}
\usepackage{amstext} 
\usepackage{graphicx}
\usepackage{subcaption}
\usepackage{amsfonts}
\usepackage{booktabs}
\usepackage{wrapfig}
\usepackage{sidecap} 
\usepackage{latexsym}
\usepackage{pdfpages}
\usepackage{epic}
\usepackage{fullpage}
\usepackage{color}
\usepackage{curves}
\usepackage{setspace}
\usepackage{amsmath}
\numberwithin{equation}{section}
\usepackage{anysize}
\usepackage{rotating}
\usepackage{enumerate} 
\usepackage{hyperref}
\usepackage{bbm}

\usepackage{graphicx}
\usepackage{algorithm} 
\usepackage{algorithmic}
\usepackage{xcolor}
\usepackage{tikz}
\usetikzlibrary{arrows.meta,positioning,shapes.geometric,shapes.misc}

\newtheorem{theorem}{Theorem}[section] 
\newtheorem{definition}[theorem]{Definition}
\newtheorem{corollary}[theorem]{Corollary}
\newtheorem{proposition}[theorem]{Proposition}
\newtheorem{lemma}[theorem]{Lemma}
\newtheorem{remark}[theorem]{Remark}
\renewcommand{\leq}{\leqslant}
\renewcommand{\geq}{\geqslant}
\numberwithin{equation}{section}

\newcommand{\pt}{\partial_t}

\title{Identification of Source Terms in the Ginzburg-Landau Equation from Final Data}
\date{}

\author{
	Roberto Morales\thanks{Chair of Computational Mathematics, DeustoTech, University of Deusto,
Avenida de las Universidades 24, Bilbao, 48007, Basque Country, Spain, e-mail: roberto.morales@deusto.es.}
	\and
	Javier Ram\'irez-Ganga\thanks{Centro de Modelamiento Matem\'atico (CMM) IRL 2807 CNRS-UChile and Departamento de Ingenier\'ia Matem\'atica (DIM), Universidad de Chile, Beauchef 851, Casilla 170-3, Correo 3, Santiago, Chile, e-mail: jramirez@dim.uchile.cl.}
	}

\begin{document}
\maketitle
\begin{abstract}
In this article, we study an inverse problem consisting in the identification of a space-time dependent source term in the Ginzburg-Landau equation from final-time observations. We adopt a weak-solution framework and analyze the associated Tikhonov functional, deriving an explicit gradient formula via an adjoint system and proving its Lipschitz continuity. We then establish existence and uniqueness results for quasi-solutions, and validate the theory with numerical experiments based on iterative methods.

\end{abstract}


\noindent {\bf MSC Classifications (2020):} 35Q56, 35R30, 49N45, 65N21. 

\section{Introduction}

This section introduces the mathematical formulation of the inverse problem and specifies the functional framework adopted throughout the paper. We also describe the model equation, its physical motivation, and the regularity assumptions required to ensure well-posedness.

Let $\Omega\subset \mathbb{R}^N$ be a bounded domain ($N\geq 1$) with boundary $\partial \Omega$ of class $C^2$ and $T>0$. Here and in the sequel, the function spaces refer to spaces of complex-valued functions unless otherwise specified. 

We consider $a,b$, $\vec{r}$ and $p$ with the following assumptions:
\begin{itemize}
    \item[{\bf (H1)}] $a,b\in W^{1,\infty}(\Omega;\mathbb{R})$ with
    \begin{align*}
        a(x)\geq a_\star >0\quad \text{ almost everywhere in }\Omega, 
    \end{align*}
    for some $a_\star$.
    \item[{\bf (H2)}] $\vec{r}=\vec{r}(x)$, $p=p(x)$ and $y_0=y_0(x)$ are complex-valued functions such that 
    \begin{align*}
        \vec{r}\in [W^{1,\infty}(\Omega)]^N,\,
        q\in L^\infty(\Omega)\text{ and }y_0\in L^2(\Omega).
    \end{align*}
\end{itemize}

Now, consider the following Ginzburg-Landau equation with Dirichlet boundary conditions:
\begin{align}
    \label{eq:intro:01}
    \begin{cases}
        \pt y -\text{div}((a(x)+ib(x))\nabla y) + \vec{r}(x)\cdot \nabla y +p(x)y=f(x,t)&\text{ in }\Omega\times (0,T),\\
        y=0&\text{ on }\partial \Omega\times (0,T),\\
        y(\cdot,0)=y_0&\text{ in }\Omega. 
    \end{cases}
\end{align}

Under the assumptions {\bf (H1)}, {\bf (H2)} and $f\in L^2(0,T;L^2(\Omega))$, there exists a unique weak solution of \eqref{eq:intro:01} with the following regularity:
\begin{align*}
    y\in C^0([0,T];L^2(\Omega)) \cap L^2(0,T;H_0^1(\Omega)).
\end{align*}

For completeness and self-containment,  we include in Appendix \ref{section:appendix:existence:uniqueness:results} general statements, together with their proofs, concerning the existence of weak and strong solutions of the equation \eqref{eq:intro:01}. 

In this article, we will consider the following {\bf Inverse Source Problem (ISP):} Identify the unknown spatial-temporal source term $f$ in the space $L^2(0,T;L^2(\Omega))$ in \eqref{eq:intro:01} from the following final time measured output:
\begin{align*}
    u_T(x):=y(x,T),\quad x\in \Omega. 
\end{align*}

Here, $y(x,T)\equiv y(x,t)\big|_{t=T}$ is a trace appropriately defined of the weak solution $y(x,t)$ of \eqref{eq:intro:01}. In addition, $u_T(x)$ represents the measured output containing a random noise. 

Inverse source problems of this type are not only mathematically rich but also arise naturally in applied sciences. In many experimental settings, direct access to distributed sources is often impossible, and only partial or final-time measurements are available. For instance, in nonlinear optics one may measure the profile of a pulse at the end of a cavity, while the internal gain or loss mechanisms driving its formation remain hidden. Similarly, in chemical or biological pattern formation, one often observes the state of the system after some transient dynamics but seeks to infer the localized heterogeneities or forcing terms that generated it. This gap between accessible measurements and hidden dynamics motivates the study of reconstruction methods capable of identifying unknown excitations from indirect data.

The choice of the Ginzburg-Landau equation as the underlying model is particularly significant. It serves as a universal amplitude equation near the onset of instability in diverse physical systems, ranging from fluid dynamics and plasma physics to optics and chemical oscillations \cite{Aranson2002,levermore1996complex,chen1998numerical,bartuccelli1990possibility,mielke2002ginzburg}. Consequently, advances in the understanding of inverse problems for this equation are not limited to a single discipline but carry implications across multiple fields. Identifying unknown sources in such a model not only enriches the theory of dissipative-dispersive PDEs but also provides a mathematical framework for interpreting experiments in which only end-point information is available.

At the mathematical level, the inverse source problem we consider is severely ill-posed. Regulation strategies and variational formulations thus become essential. Our work adopts this perspective, combining the theory of weak solutions, adjoint-based gradient formulas, and Tikhonov regularization to rigorously justify reconstruction algorithms and test their performance through numerical simulations.


\subsection{Literature}

We now review some representative contributions on the complex Ginzburg-Landau equation and its variants, emphasizing their role as canonical amplitude equations and discussing prior work related to inverse or identification problems.

The complex Ginzburg-Landau (GL) equation is the canonical amplitude equation near a Hopf or Turing-Hopf bifurcation. In its cubic form, it reads
\begin{align*}
    \pt A=(\mu+i\omega)A + (1+c_1i)\Delta A -(1+c_3 i)|A|^2 A,
\end{align*}
while the cubic-quintic extension adds a saturating nonlinearity:
\begin{align*}
    \pt A=(\mu +\omega i)A +(1+c_1i)\Delta A -(1+c_3 i)|A|^2 A -(\nu +c_5i)|A|^4A,
\end{align*}
with real parameters $\mu,\nu,c_j$ ($j=1,3,5$) that encode growth, dispersion/diffusion, and nonlinear gain/loss. The cubic model is the universal normal form near onset and captures phase diffusion, plane-wave selection and Benjamin-Feir/Newell instabilities. On the other hand, the quintic term become relevant farther from threshold to model saturation and complex dissipative structures (e.g. localized pulses). See the reviews \cite{Cross1993} and \cite{Aranson2002}, which also discuss higher-order variants and parameter scalings.

The cubic and cubic-quintic complex GL equations serve as canonical models of pattern formation and nonlinear wave dynamics in nonequilibrium systems. These models describe the slow complex envelope $A(x,t)$ of oscillatory modes near stability. In hydrodynamics and Rayleigh-B\'enard convection, the cubic GL equation explains the emergence of roll patterns, travelling waves, and their stability domains \cite{Newell1969} and \cite{Cross1993}. In nonlinear optics, it governs mode-locked lasers, pulse propagation in fiber cavities, and the formation of dissipative solutions. In this case, the cubic-quintic equation extension accounts for gain saturation and quintic nonlinear effects that are essential for bounded amplitude states \cite{SotoCrespo1997}, \cite{Akhmediev2005}. Similarly, in chemical oscillations and biological excitable media, it models spiral waves, defect turbulence, and spatiotemporal chaos, highlighting its universality across disciplines \cite{Aranson2002}.

From a mathematical viewpoint, the GL equations combine parabolic smoothing (via the Laplacian operator) with dispersive phase dynamics (via imaginary coefficients), which generates a delicate interplay between diffusion, dispersion, and nonlinearity \cite{Akhmediev2005,Cross1993, levermore1996complex, mielke2002ginzburg}. This hybrid character produces a wide spectrum of phenomena: existence of plane-wave solutions, modulational (Benjamin-Feir/Newell) instabilities, turbulence, and complex coherent structures \cite{Akhmediev2005, Cross1993, Newell1969, SotoCrespo1997}. The cubic equation alone already yields chaotic attractors, defect-mediated turbulence, and spiral wave breakup \cite{Aranson2002,bartuccelli1990possibility, SotoCrespo1997}, while the quintic term stabilizes or destabilizes localized patterns depending on parameter regimes \cite{Akhmediev2005,SotoCrespo1997}. The GL equations have been used to study attractors, bifurcations, and long-time dynamics of nonlinear PDEs, serving as a testbed for methods in dissipative systems, control, and inverse problems \cite{chen1998numerical,mielke2002ginzburg,santos2019insensitizing,borzi2005analysis,junge2000synchronization}. 

Recovering sources and parameters in GL-type models is crucial in practice: in optics, it enables diagnosing distributed gain/loss or saturable absorbers in cavity models from end-of-pulse measurements; in chemical/biological pattern formation, it helps infer spatially localized forcing or heterogeneities driving wave patterns; in fluid and plasma contexts, it allows one to back-out effective forcing or feedback from limited snapshots near transitions to turbulence. Because the GL framework is a universal amplitude model, successful inversion translates across disciplines, providing interpretable maps of where and how energy is injected or dissipated. From a computational perspective, adjoint-based Tikhonov schemes scale to high-dimensional discretizations and accommodate realistic noise models \cite{Cross1993}, \cite{Aranson2002}, \cite{MR2516528}.

Within this broad literature, inverse problems for GL-type models remain relatively less explored. Classical studies on inverse source problems for parabolic or Schr\"odinger-type equations \cite{hasanouglu2021introduction,hasanov2007simultaneous,GarciaOssesTapia+2013+755+779,chorfi2025identification} provide methodological foundations, yet the adaptation to GL equations, which combine diffusion, dispersion, and nonlinearity, poses unique analytical and numerical challenges.

We also note that other GL-type models, such as those with dynamic boundary conditions, have recently attracted attention in the context of controllability \cite{carreno2025local}. While these studies highlight the richness of the GL framework under more complex boundary interactions, the analysis of inverse problems in such settings remains largely open. Our contribution is thus complementary, as we address the identification of space-time dependent source terms for the standard GL equation, providing a first step toward broader inverse problem formulations in extended GL models.

It is worth emphasizing that the linear GL equation can also be expressed as a system of two coupled real-valued PDE equations, corresponding to the real and imaginary parts of the complex amplitude. In this formulation, the coupling terms are of second order and involve cross-derivative operators, which profoundly alter the analytical structure of the system. As a consequence, classical techniques from the theory of control and inverse problems for scalar parabolic equations (based on Carleman estimates and observability inequalities \cite{fursikov1996carleman}) cannot be directly applied. In particular, the presence of complex coefficients and dispersive-type interactions requires adapted Carleman weights, refined energy estimates, and the development of new stability results tailored to mixed dissipative-dispersive operators.  

The present work is intended as a first step toward a comprehensive theory of inverse problems for GL-type equations, focusing here on the standard setting with Dirichlet boundary conditions.

\subsection{The input-output operator}
To formalize the inverse problem, we introduce the input-output operator mapping the source term to the final-state of the system. This operator provides a compact representation of the forward model and allows for a variational reformulation of the inverse problem.

Let $\mathcal{U}\subset L^2(0,T;L^2(\Omega))$ be a non-empty set, which is supposed to be bounded, closed and convex. In addition, we define the operator $\Psi: \mathcal{U}\to L^2(\Omega)$ in the following way:
\begin{align*}
    \Psi(f)=y(\cdot,T)\quad f\in \mathcal{U},
\end{align*}
where $y=y(x,t)$ is the solution of \eqref{eq:intro:01} associated to the source term $f$.

We point out that, the {\bf (ISP)} can be formulated in terms of $\Psi$ as the following functional equation:
\begin{align}
    \label{eq:Input-output:op}
    \Psi(f)=u_T(x),\quad x\in \Omega,\quad  f\in \mathcal{U}. 
\end{align}

We emphasize that, due to measurement errors in the output $u_T$, exact equality in the equation \eqref{eq:Input-output:op} cannot be satisfied in general. Moreover, we have the following result
\begin{proposition}
    \label{Proposition:Phi:compact}
    The input-output operator $\Psi:L^2(0,T;L^2(\Omega))\to L^2(\Omega)$ is compact.
\end{proposition}

\begin{proof}
    Let $(f_k)_{k\in \mathbb{N}}$ be a bounded sequence in $L^2(\Omega\times (0,T))$. Then, by Proposition \ref{proposition:existence:strong:solutions}, the sequence of associated weak solutions $(y(\cdot,,\cdot,f_k))_{k\in \mathbb{N}}$ is bounded in $C^0([0,T];H_0^1(\Omega))$. In particular, the sequence $(Y_k)_{k\in \mathbb{N}}$ given by $Y_k:=y(\cdot,T,f_k)$ is bounded in $H_0^1(\Omega)\cap H^2(\Omega)$. Thus, using the Sobolev-Gagliardo-Nirenberg compact embedding $H_0^1(\Omega) \hookrightarrow L^2(\Omega)$, there exists a subsequence of $ (Y_k)_{k\in \mathbb{N}}$ which converges strongly in $L^2(\Omega)$. This implies that the input-output operator $\Psi$ is compact and the proof of Proposition  \ref{Proposition:Phi:compact} is finished. 
\end{proof}

In view of the Proposition \ref{Proposition:Phi:compact}, it is evident that the inverse problem {\bf (ISP)} is ill-posed in the sense of Hadamard. For this reason, one needs to introduce the functional $\mathcal{J}:\mathcal{U}\to \mathbb{R}$ by
\begin{align}
    \label{def:functional:J}
    \mathcal{J}(f):=\frac{1}{2}\int_\Omega |\Psi(f)-u_T|^2\,dx,
\end{align}
and reformulate {\bf (ISP)} in terms of the quasi-solution method, i.e., we minimize the following extremal problem
\begin{align*}
    \mathcal{J}(f^\star)=\inf_{f\in \mathcal{U}} \mathcal{J}(f).
\end{align*}

Since the operator $\Psi$ is compact, small perturbations in the data may cause large variations in the reconstructed source. Therefore, regularization becomes essential to obtain stable approximations. In this context, it is customary to consider a regularized Tikhonov version of the functional $J$ in \eqref{def:functional:J}. For $\epsilon>0$, we introduce the regularized functional
\begin{align}\label{regfunct}
    \mathcal{J}_\epsilon (f):=\frac{1}{2}\int_\Omega |\Psi(f)-u_T|^2\,dx + \frac{\epsilon}{2}\int_0^T\int_\Omega |f|^2\,dx\,dt ,\quad f\in \mathcal{U}. 
\end{align}

\subsection{Outline}
The rest of the paper is as follows. In Section \ref{section:Frechet:formula}, we focus on the properties of $J$, i.e., a detailed characterization of its Fr\'echet derivative via a suitable adjoint system. In Section \ref{section:Existence:ISP}, we give sufficient conditions for the existence and uniqueness of quasi-solutions to {\bf (ISP)}. In Section \ref{section:Numerical:experiments}, we validate our theoretical results by some numerical experiments to reconstruct source terms for 1-D and 2-D case. Finally, in Section \ref{section:Summary:perspectives} we give additional comments concerning the theoretical and numerical results obtained in this article.

\section{Fr\'echet differentiability and gradient formula}
\label{section:Frechet:formula}


The differentiability of the misfit functional is central to the design of iterative reconstruction methods. Establishing explicit gradient representations not only ensures theoretical well-posedness of descent algorithms but also provides a concrete path to numerical implementation. We therefore begin by characterizing the gradient of the functional via a suitable adjoint system.

\begin{proposition}
    \label{proposition:Frechet:formula}
    Consider the assumptions {\bf (H1)} and {\bf (H2)}. Then, the functional $\mathcal{J}:\mathcal{U}\to \mathbb{R}$ is Fr\'echet differentiable and its gradient at each $f\in \mathcal{U}$ is given by 
    \begin{align}
        \label{eq:frechet:gradient:formula}
        \mathcal{J}'(f)=\phi,\quad f\in \mathcal{U},
    \end{align}
    where $\phi$ is the unique weak solution of the following adjoint system 
    \begin{align}
        \label{problem:Frechet:dif}
        \begin{cases}
            -\pt \phi -\text{div}((a(x)-ib(x))\nabla \phi) - \overline{\vec{r}(x)}\cdot \nabla \phi  +(\overline{p(x)}- \overline{\text{div}(\vec{r})(x))}\phi =0&\text{ in }\Omega\times (0,T),\\
            \phi=0&\text{ on }\partial\Omega \times (0,T),\\
            \phi(\cdot,T)=y(\cdot,T;f)-u_T&\text{ in }\Omega. 
        \end{cases}
    \end{align}
\end{proposition}

The adjoint formulation obtained in Proposition \ref{proposition:Frechet:formula} mirrors the standard duality principles commonly used in PDE-constrained optimization. Its importance lies in reducing the computational cost: instead of computing directional derivatives for each perturbation, the gradient can be evaluated by solving a single adjoint problem, making the approach scalable to high-dimensional discretizations.

\begin{proof}
    let us consider $f,\delta f\in \mathcal{U}$ such that $f+\delta f\in \mathcal{U}$. Then, our task is to compute the difference
    \begin{align*}
        \delta \mathcal{J}(f):=&\mathcal{J}(f+\delta f)-\mathcal{J}(f)\\
        =&\frac{1}{2}\int_\Omega \left(|y(\cdot,T;f+\delta f)-u_T|^2 - |y(\cdot,T;f)-u_T|^2  \right)\,dx.
    \end{align*}

Using the complex identity
\begin{align*}
    \frac{1}{2}(|x-z|^2-|y-z|^2)=\Re [(x-y)\overline{(y-z)}] +\frac{1}{2} |x-y|^2,\quad \forall x,y,z\in \mathbb{C},
\end{align*}
we have 
\begin{align}
    \nonumber 
    \delta \mathcal{J}(f)=&\Re \int_\Omega (y(\cdot,T;f+\delta f)-y(\cdot,T;f)) \overline{y(\cdot,T;f)-u_T}\,dx \\
    \nonumber 
    &+\frac{1}{2}\int_\Omega |y(\cdot,T;f+\delta f)-y(\cdot,T;f)|^2\,dx \\
    \label{eq:Frechet:for:01} 
    =&-\Re \int_\Omega \delta y(\cdot,T;f)\overline{\phi (\cdot,T;f)}\,dx +\frac{1}{2}\int_\Omega |\delta y(\cdot,T;f)|^2\,dx 
\end{align}
where $\phi(\cdot,\cdot,f)$ is the weak solution of \eqref{problem:Frechet:dif} (and thanks to Proposition \ref{proposition:existence:weak:solutions}, $\phi(\cdot,T,f)=y(\cdot,T;f)-u_T$ in $L^2(\Omega)$) and $\delta y$ is the solution of the following \textit{sensitivity problem}
\begin{align}
    \label{eq:sensitivity:problem}
    \begin{cases}
        \pt \delta y -\text{div}((a(x)+b(x)i)\nabla \delta y)+\vec{r}(x)\cdot \nabla \delta y +  p(x)\delta y=\delta f(x,t)&\text{ in }\Omega\times (0,T),\\
        \delta y=0&\text{ on }\partial \Omega\times (0,T),\\
        \delta y(\cdot,0)=0&\text{ in }\Omega. 
    \end{cases}
\end{align}

Observe that by Proposition \ref{proposition:existence:strong:solutions}, $\delta y$ is indeed a strong solution of \eqref{eq:sensitivity:problem}. Then, integration by parts implies that \eqref{eq:Frechet:for:01} can be written as 
\begin{align}
    \label{form:frechet:grad:var:for}
    \delta \mathcal{J}(f):=\Re \int_0^T\int_\Omega \delta f\overline{\phi}\,dx\,dt +\frac{1}{2}\int_\Omega |\delta y(\cdot,T;f)|^2\,dx.
\end{align}

Now, identity \eqref{form:frechet:grad:var:for} implies the assertion \eqref{eq:frechet:gradient:formula}.
\end{proof}

\begin{lemma}
    \label{Lemma:Lipschitz:gradient}
    Under the assumptions {\bf (H1)} and {\bf (H2)}, the Fr\'echet gradient $\mathcal{J}'$ is Lipschitz continuous in $\mathcal{U}$, i.e., there exists a constant $M>0$ such that for all $f,\delta f\in \mathcal{U}$ such that $f+\delta f\in \mathcal{U}$, the following estimate holds:
    \begin{align*}
        |\mathcal{J}'(f+\delta f) - \mathcal{J}'(f)|\leq M\|\delta f\|_{L^2(\Omega\times (0,T))}.
    \end{align*}
\end{lemma}

We point out that the Lipschitz continuity of the gradient guarantees the stability of iterative schemes such as conjugate gradient or quasi-Newton methods, making the result crucial for practical computations. For more details, see Section \ref{section:Numerical:experiments}.

The Lipschitz continuity of the gradient ensures that iterative minimization schemes converge under standard step-size rules. In practice, this property prevents instability in the optimization process and justifies the use of line-search or trust-region strategies in the reconstruction algorithm.

\begin{proof}
    Let us fix $f,\delta f\in \mathcal{U}$ such that $f+\delta f\in \mathcal{U}$. Then, the function $\delta \phi=\delta \phi (\cdot,\cdot,f)$ is the solution of
    \begin{align}
        \label{eq:Lipschitz:01}
        \begin{cases}
            -\pt \delta \phi -\text{div}((a(x)-ib(x))\nabla \phi) - \overline{\vec{r}(x)}\cdot \nabla \phi +(\overline{p(x)}-\overline{\text{div}(\vec{r}(x))}\phi =0&\text{ in }\Omega\times (0,T),\\
            \delta \phi =0&\text{ on }\partial \Omega \times (0,T),\\
            \delta \phi(\cdot,T)=\delta y(\cdot,T;f)&\text{ in }\Omega, 
        \end{cases}
    \end{align}
    where $\delta y$ is the solution of \eqref{problem:Frechet:dif} associated to $\delta f$. We remark that, by Proposition \ref{proposition:existence:strong:solutions}, $\delta \phi$ is a strong solution of \eqref{eq:Lipschitz:01}. Thus, by Proposition \ref{proposition:Frechet:formula} we have, for some constant $C>0$:
    \begin{align*}
        |\mathcal{J}'(f+\delta f)-\mathcal{J}'(f)|= \|\delta \phi\|_{L^2(\Omega\times (0,T))}\leq C\|\delta y(\cdot,T;f)\|_{L^2(\Omega)}\leq C\|\delta f\|_{L^2(0,T;L^2(\Omega))},
    \end{align*}
    where we have used the continuity of the strong solutions applied to $\delta y$ respect to the data. This proves the assertion of the Lemma \ref{Lemma:Lipschitz:gradient}.  
\end{proof}

\section{Existence and uniqueness of the solution to (ISP)}
\label{section:Existence:ISP}


Having established the differentiability of the functional, we now turn to the fundamental question of solvability of the {\bf (ISP)}. Existence of quasi-solutions ensures that our variational formulation is mathematically meaningful, while uniqueness conditions determine whether the reconstruction is well-posed or only identifiable up to certain ambiguities.

Before going further, we prove the following
\begin{lemma}
    \label{lemma:continuity:solution:source}
    Let $y_0\in L^2(\Omega)$ and let $y$ be the weak solution of \eqref{eq:intro:01} corresponding to $f$. Then, the solution map $f\mapsto y$ is continuous from $L^2(0,T;L^2(\Omega))$ to $C^0([0,T];L^2(\Omega)) \cap L^2(0,T;H_0^1(\Omega))$.
\end{lemma}
\begin{proof}
    By density, it is sufficient to consider $y_0\in H_0^1(\Omega)$. Let $\delta f$ be a small variation of $f$ such that $f^\delta =f+\delta f\in \mathcal{U}$. Consider $\delta y=y^\delta -y$, where $y^\delta$ is the weak solution of \eqref{eq:intro:01} corresponding to $f^\delta$. Then, by Proposition \ref{proposition:existence:weak:solutions}, we know that 
    \begin{align*}
        \delta y=y(\cdot,\cdot;f^\delta) - y(\cdot,\cdot,f)\in C^0([0,T];L^2(\Omega)) \cap L^2(0,T;H_0^1(\Omega))
    \end{align*}
    and satisfy the following system 
    \begin{align}
        \label{eq:delta:y}
        \begin{cases}
            \pt \delta y -\text{div}((a(x)+b(x)i)\nabla \delta y) +\vec{r}(x)\cdot \nabla \delta y + p(x)\delta y=\delta f(x,t)&\text{ in }\Omega\times (0,T),\\
            \delta y=0&\text{ on }\partial \Omega\times (0,T),\\
            \delta y(\cdot,0)=0&\text{ in }\Omega. 
        \end{cases}
    \end{align}

    Therefore, the continuity of the solution $\delta y$ of \eqref{eq:delta:y} respect to the data
    \begin{align*}
        \|\delta y\|_{C^0([0,T]; L^2(\Omega))} + \|\delta y\|_{L^2(0,T;H_0^1(\Omega))} \leq C\|\delta f\|_{L^2(0,T;L^2(\Omega))}
    \end{align*}
    proves the assertion of Lemma \ref{lemma:continuity:solution:source}.
\end{proof}

In particular, using the gradient formula \eqref{eq:frechet:gradient:formula} and the estimates associated to the sensitivity problem \eqref{eq:sensitivity:problem}, we will prove the monotonicity of the derivative $\mathcal{J}'$. In particular, this estimate implies the convexity of the functional $\mathcal{J}$. Consequently, we have the following existence result:

\begin{corollary}
    The Tikhonov functional $\mathcal{J}$ is continuous and convex on the subset $\mathcal{U}$. Then, there exists a minimizer $f^\star\in \mathcal{F}$ such that 
    \begin{align*}
        \mathcal{J}(f^\star)=\min_{f\in \mathcal{U}} \mathcal{J}(f). 
    \end{align*}

    Since the strict convexity of $\mathcal{J}$ is characterized by the strict monotonicity of $\mathcal{J}'$. Besides, after some computations, we have the following sufficient condition for uniqueness.
\end{corollary}

The convexity of the functional guarantees the existence of minimizers and provides a favorable landscape for optimization. In contrast to non-convex problems, here one avoids the proliferation of spurious local minima, which simplifies the design of reconstruction algorithms.

    \begin{lemma}
        If the positivity condition
        \begin{align} 
        \label{eq:lemma3.3}
        \int_\Omega |\delta y(\cdot,T;f)|^2\,dx >0,\quad \forall f\in \mathcal{V},
        \end{align} 
    holds on a closed convex subset $\mathcal{V}\subset \mathcal{U}$, then the problem {\bf (ISP)} admits at most one solution in $\mathcal{V}$.
    \end{lemma}

Condition \eqref{eq:lemma3.3} reflects the observability of perturbations in the final state. Intuitively, if every non-trivial source produces a distinguishable final profile, uniqueness follows. Such positivity conditions are reminiscent of unique continuation principles in PDEs and are closely tied to the richness of the data.

\section{Numerical results}\label{section:Numerical:experiments}

In this section we present the numerical scheme used to recover the complex-valued source factor \(q(x,y)\) in the multiplicative model
\begin{equation*}
    f(x,y,t)=q(x,y)\,g(t),
\end{equation*}
and report representative results for two--dimensional problems posed on rectangular domains. All computations are performed in \verb|Python| and combine second--order finite differences in space with a Crank--Nicolson (CN) discretization in time for both the forward and adjoint equations. To streamline the exposition and match the implementation, we consider constant coefficients \(a>0\), \(b\in\mathbb{R}\), and the model
\begin{align*}
    \begin{cases}
        \partial_t y - (a+i\,b)\,\Delta y + p\,y= f & \text{in } \Omega\times (0,T),\\[0.2em]
        y=0 & \text{on } \partial \Omega\times (0,T),\\[0.2em]
        y(\cdot,0)=y_0 & \text{in } \Omega,
    \end{cases}
\end{align*}
with \(\Omega=(0,L_{x_1})\times(0,L_{x_2})\subset\mathbb{R}^2\). The spatio--temporal modulation \(g\) is known and prescribed. In all experiments we enforce homogeneous Dirichlet boundary conditions on \(u\) and set \(q\equiv 0\) and \(g\equiv 0\) on \(\partial\Omega\), which implies that \(f\) vanishes on the boundary and is supported strictly in the interior.

\subsection{Discretization and numerical implementation}

We use second--order finite differences on uniform grids
\[
    x_{1,i}=i\,\Delta x_1,\quad i=0,\dots,N_{x_1}, 
    \qquad
    x_{2,j}=j\,\Delta x_2,\quad j=0,\dots,N_{x_2},
\]
with mesh sizes \(\Delta x_1=L_{x_1}/N_{x_1}\) and \(\Delta x_2=L_{x_2}/N_{x_2}\). Homogeneous Dirichlet data are imposed on \(\partial\Omega\); thus the unknowns live on the interior grid
\[
  \big\{(x_{1,i},x_{2,j}):\, i=1,\dots,N_{x_1}-1,\ j=1,\dots,N_{x_2}-1\big\}.
\]
Let \(\Delta_{h_{x_1}}\in\mathbb{R}^{(N_{x_1}-1)\times(N_{x_1}-1)}\) and
\(\Delta_{h_{x_2}}\in\mathbb{R}^{(N_{x_2}-1)\times(N_{x_2}-1)}\) denote the 1D discrete Dirichlet Laplacians \cite{MR2378550},
\[
\Delta_{h_{x_1}}=\frac{1}{\Delta x_1^{2}}\operatorname{tridiag}(1,-2,1), 
\qquad
\Delta_{h_{x_2}}=\frac{1}{\Delta x_2^{2}}\operatorname{tridiag}(1,-2,1).
\]
With lexicographic ordering, the 2D Laplacian is the Kronecker sum
\begin{equation}\label{eq:Lapl2D}
\Delta_h = I_{x_2}\otimes \Delta_{h_{x_1}} + \Delta_{h_{x_2}}\otimes I_{x_1}
\in\mathbb{R}^{m\times m},\qquad m=(N_{x_1}-1)(N_{x_2}-1),
\end{equation}
where \(I_{x_1}\) and \(I_{x_2}\) are identity matrices of sizes \(N_{x_1}-1\) and \(N_{x_2}-1\), respectively, and \(\otimes\) denotes the Kronecker product.

For the time grid we set \(t^n=n\,\Delta t\), \(n=0,\dots,N_t\), with \(\Delta t=T/N_t\).
Let \(A:=(a+i\,b)\,\Delta_h+pI_{x_1\times x_2}\). Following \cite{MR2002152}, the CN update for the interior degrees of freedom reads
\begin{equation}\label{eq:CN-forward}
M_-\,y^{n+1} = M_+\,y^{n} + \Delta t\, f^{n},
\qquad
M_\pm = I \pm \tfrac{\Delta t}{2}\,A,
\end{equation}
with \(y^0\) given by the samples of \(y_0\) on the interior grid and \(I\in\mathbb{C}^{m\times m}\). The forcing is treated by a left-point (rectangle) rule in time, consistent with our choice of quadrature below. Dirichlet boundary values vanish identically, hence they do not appear in~\eqref{eq:CN-forward}. The matrix \(M_-\) is factorized once (sparse LU) and reused at all time steps.

We approximate space--time integrals using mass--lumped rectangle rules on the interior grid. Define the spatial inner product
\[
\langle \xi,\eta\rangle_{h} := \mathrm{Re}(\xi^*\eta)\,\Delta x_1\,\Delta x_2,
\quad \xi,\eta\in\mathbb{C}^{m},
\]
and the space--time inner product
\[
\langle\!\langle X,Y\rangle\!\rangle_{h,t}
:=\sum_{n=0}^{N_t-1}\langle X^n,Y^n\rangle_h\,\Delta t,
\qquad
\|X\|_{h,t}^2:=\langle\!\langle X,X\rangle\!\rangle_{h,t}.
\]
Given target samples \(v_h^n\) on the interior grid, we consider the \emph{terminal tracking} functional
\begin{equation}\label{eq:J-discrete}
J_h(f) = \tfrac{1}{2}\,\langle y^{N_t}-v_h^{N_t},\,y^{N_t}-v_h^{N_t}\rangle_{h}
+ \tfrac{\epsilon}{2}\,\langle\!\langle f,f\rangle\!\rangle_{h,t},
\qquad \epsilon>0.
\end{equation}

The discrete adjoint associated with \eqref{eq:CN-forward} and \eqref{eq:J-discrete} is the time--reversed recursion
\begin{equation}\label{eq:CN-adjoint}
M_-^{*}\,\phi^{n} = M_+^{*}\,\phi^{n+1},\quad n=N_t-1,\dots,0,
\qquad
\phi^{N_t} = y^{N_t}-v_h^{N_t},
\end{equation}
where \((\cdot)^*\) denotes the Hermitian transpose. Since \(A=(a+i\,b)\Delta_h\), we have \(A^*=(a-i\,b)\Delta_h\), and thus
\(M_\pm^{*}=I\pm\tfrac{\Delta t}{2}\,A^*\).
As in the forward solve, \(M_-^*\) is factorized once and reused \cite{MR2516528}.

\begin{remark}
    Since \(\Re(a+i b)=a>0\) and \(\Delta_h\) is symmetric negative definite under Dirichlet conditions, the operator \(A\) is strictly dissipative. The CN scheme is \(A\)-stable and second--order accurate for the homogeneous part. With the left-point treatment of the forcing, the method remains unconditionally stable; using a trapezoidal rule for \(f\) restores second--order temporal accuracy. In all cases, the spatial discretization is second order.
\end{remark}

By standard Lagrangian arguments (or by differentiating \eqref{eq:CN-forward} directly), the discrete gradient of \(J_h\) with respect to the control at time level \(n\) is
\begin{equation}\label{eq:grad}
\frac{\partial J_h}{\partial f^{n}} =  \phi^{n+1} + \varepsilon\, f^{n}
\quad\in\mathbb{C}^{m},\qquad n=0,\dots,N_t-1,
\end{equation}
so that \(\nabla J_h(f)=\big( z^{n+1}+\epsilon\,f^{n}\big)_{n=0}^{N_t-1}\).

We minimize \(J_h\) using the Polak--Ribi\`ere\(^+\) nonlinear conjugate gradient (NCG) method with Armijo backtracking and an optional \(L^2\) projection onto a closed ball \(\{f:\|f\|_{h,t}\le \rho\}\). This follows the standard framework of PDE-constrained optimization and is summarized in Algorithm~\ref{alg:NCG} \cite{MR2244940,MR1740963,MR2516528}. To avoid notational clash with the modulation \(g\), we denote the gradient at iteration \(k\) by \(r_k:=\nabla J_h(f_k)\).

\begin{algorithm}
\begin{algorithmic}[1]
    \STATE Choose an initial guess $f_0$, set $k=0$.
    \WHILE{$\|r_k\|_{h,t} \geq \tau$ and $k<k_{\max}$}
        \STATE Solve the forward problem \eqref{eq:CN-forward} to obtain $\Psi(f_k)$, then the adjoint \eqref{eq:CN-adjoint} to obtain $\phi(f_k)$.
        \STATE Assemble the gradient $r_k^n = \phi^{n+1}(f_k) + \epsilon\,f_k^n$ for $n=0,\dots,N_t-1$.
        \STATE Set the search direction
        \[
        d_k=
        \begin{cases}
        -r_k, & \text{if } k=0 \text{ or } \langle\!\langle r_k,d_{k-1}\rangle\!\rangle_{h,t}\ge 0,\\
        -r_k + \beta_k d_{k-1}, & \text{otherwise},
        \end{cases}
        \]
        \[
        \beta_k=\max\!\left\{0,\dfrac{\langle\!\langle r_k-r_{k-1},\,r_k\rangle\!\rangle_{h,t}}{\langle\!\langle r_{k-1},\,r_{k-1}\rangle\!\rangle_{h,t}+10^{-30}}\right\}.
        \]
        \STATE Perform Armijo backtracking to find $\alpha>0$ such that
        \[
        J_h(f_k+\alpha d_k) \leq J_h(f_k) + c\,\alpha\,\langle\!\langle r_k,d_k\rangle\!\rangle_{h,t},
        \qquad c\in(0,1).
        \]
        \STATE Update $f_{k+1}=f_k+\alpha d_k$ and, if needed, project onto the $L^2$ ball; set $k\gets k+1$.
    \ENDWHILE
\end{algorithmic}
\caption{Polak--Ribi\`ere$^+$ NCG with adjoint--based gradients and Armijo line search.}
\label{alg:NCG}
\end{algorithm}

\subsection{Numerical experiments}

We now present a sequence of test cases designed to assess reconstruction accuracy, robustness to discretization parameters, and sensitivity to noise.

Unless stated otherwise, we set \(L_x=L_y=T=1\), choose \((a,b)=(36\times 10^{-4},\,15\times 10^{-4})\) and \(p=0.2 + i\,0.1\), employ uniform grids with \((N_x,N_y,N_t)=(100,100,70)\), and take the initial condition \(u_0(x,y)=\sin(\pi x)\sin(\pi y)\). The Armijo line search uses \(c=10^{-3}\), an initial step \(\alpha_0\), and a backtracking factor of \(1/2\). We restart the Polak--Ribi\`ere\(^+\) method every five iterations or whenever \(\langle\!\langle r_k,d_k\rangle\!\rangle_{h,t}\ge 0\), where \(r_k=\nabla J_h(f_k)\) denotes the gradient.

\subsubsection{Example 1 (Smooth real-valued source)}

Our first example considers a smooth real-valued source, serving as a baseline for algorithm verification.

As a benchmark, we consider the smooth real-valued source
\begin{equation*}
    q(x,y)=\sin(2\pi x)\sin(2\pi y).
\end{equation*}
Here \(\Psi(f)\) denotes the forward solution evaluated at the final time \(t=T\), and \(u_T\) the measured final state. All error metrics below are relative discrete \(L^2\) errors, using the norm induced by \(\langle\cdot,\cdot\rangle_h\).

Table~\ref{table1} reports the performance as the stopping tolerance \(\tau\) varies for fixed regularization \(\epsilon=10^{-5}\). As \(\tau\) decreases, the iteration count increases and both the data misfit \(\|\Psi(f)-u_T\|_h/\|u_T\|_h\) and the source error \(\|q_{\mathrm{rec}}-q\|_h/\|q\|_h\) decrease, as expected for a tighter termination criterion.

\begin{table}[htbp]
\caption{Effect of the stopping tolerance with \(\epsilon = 10^{-5}\).}
\label{table1}
\centering
\begin{tabular}{l | l l l}
\toprule
 \(\tau\) & iterations & \(\dfrac{\Vert \Psi(f)-u_T\Vert_{h}^{2}}{\Vert u_T\Vert_{h}^{2}}\) & \(\dfrac{\Vert q_{\mathrm{rec}}-q\Vert_{h}^{2}}{\Vert q \Vert_{h}^{2}}\) \\
\midrule
\(10^{-3}\) & 36 & \(1.2577\times 10^{-3}\) & \(1.9513\times 10^{-3}\) \\
\(10^{-4}\) & 49 & \(1.2415\times 10^{-4}\) & \(1.9262\times 10^{-4}\) \\
\(10^{-5}\) & 62 & \(2.0756\times 10^{-5}\) & \(3.2202\times 10^{-5}\) \\
\(10^{-6}\) & 76 & \(8.2193\times 10^{-6}\) & \(1.2751\times 10^{-5}\) \\
\bottomrule
\end{tabular}
\end{table}

Table~\ref{table2} shows that the algorithm is robust with respect to mesh and time refinements: the iteration counts remain essentially unchanged and the relative errors are stable across \((N_x,N_y,N_t)\).

\begin{table}[htbp]
\caption{Mesh and time sensitivity for \(\epsilon = 10^{-5}\) and \(\tau = 10^{-5}\).}
\label{table2}
\centering
\begin{tabular}{l | l l l}
\toprule
 \((N_x,N_y,N_t)\) & iterations & \(\dfrac{\Vert \Psi(f)-u_T\Vert_{h}^{2}}{\Vert u_T\Vert_{h}^{2}}\) & \(\dfrac{\Vert q_{\mathrm{rec}}-q\Vert_{h}^{2}}{\Vert q \Vert_{h}^{2}}\) \\
\midrule
\((35,35,70)\)   & 63 & \(5.6624\times 10^{-6}\) & \(8.7835\times 10^{-6}\) \\
\((50,50,70)\)   & 63 & \(5.0117\times 10^{-6}\) & \(7.7749\times 10^{-6}\) \\
\((100,100,70)\) & 62 & \(2.0756\times 10^{-5}\) & \(3.2202\times 10^{-5}\) \\
\((100,100,150)\)& 62 & \(2.0173\times 10^{-5}\) & \(3.1298\times 10^{-5}\) \\
\bottomrule
\end{tabular}
\end{table}

Table~\ref{table3} illustrates the classical behavior of Tikhonov regularization as \(\epsilon\) varies (with fixed \(\tau=10^{-5}\)). In the noise-free setting, smaller \(\epsilon\) reduces the data misfit and reconstruction error, while larger \(\epsilon\) enforces stronger smoothing at the expense of higher bias.

\begin{table}[htbp]
\caption{Effect of \(\epsilon\) with \((N_x,N_y,N_t)=(100,100,70)\) and \(\tau = 10^{-5}\).}
\label{table3}
\centering
\begin{tabular}{l | l l l}
\toprule
 \(\epsilon\) & iterations & \(\dfrac{\Vert \Psi(f)-u_T\Vert_{h}^{2}}{\Vert u_T\Vert_{h}^{2}}\) & \(\dfrac{\Vert q_{\mathrm{rec}}-q\Vert_{h}^{2}}{\Vert q \Vert_{h}^{2}}\) \\
\midrule
 \(10^{-3}\) & 63 & \(6.8617\times 10^{-4}\) & \(1.0646\times 10^{-3}\) \\
 \(10^{-4}\) & 62 & \(8.3959\times 10^{-5}\) & \(1.3026\times 10^{-4}\) \\
 \(10^{-5}\) & 62 & \(2.0756\times 10^{-5}\) & \(3.2202\times 10^{-5}\) \\
 \(10^{-6}\) & 62 & \(1.4436\times 10^{-5}\) & \(2.2396\times 10^{-5}\) \\
\bottomrule
\end{tabular}
\end{table}

Figures~\ref{fig1}--\ref{fig2} compare the reconstructed and measured final states, displaying real and imaginary parts of \(\Psi(f)\) and \(u_T\). The close agreement is consistent with the value of the tracking functional (cf. \eqref{eq:J-discrete}).

\begin{figure}[htbp]
    \centering
    \begin{subfigure}[b]{0.32\textwidth}
        \centering
        \includegraphics[width=\textwidth]{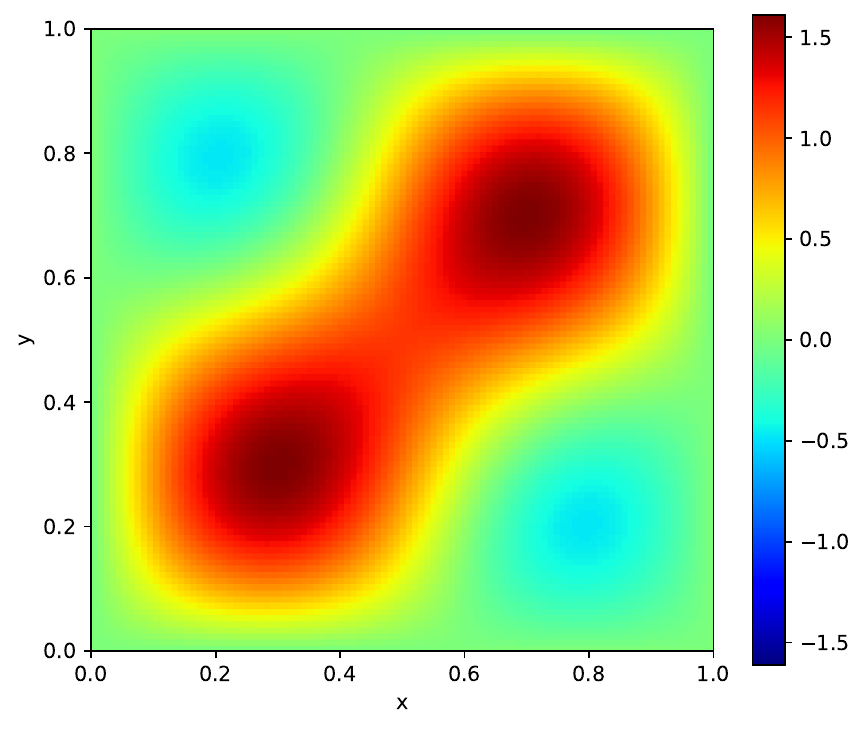}
        \caption{$\Re\big(\Psi(f)\big)$}
    \end{subfigure}
    \hfill
    \begin{subfigure}[b]{0.32\textwidth}
        \centering
        \includegraphics[width=\textwidth]{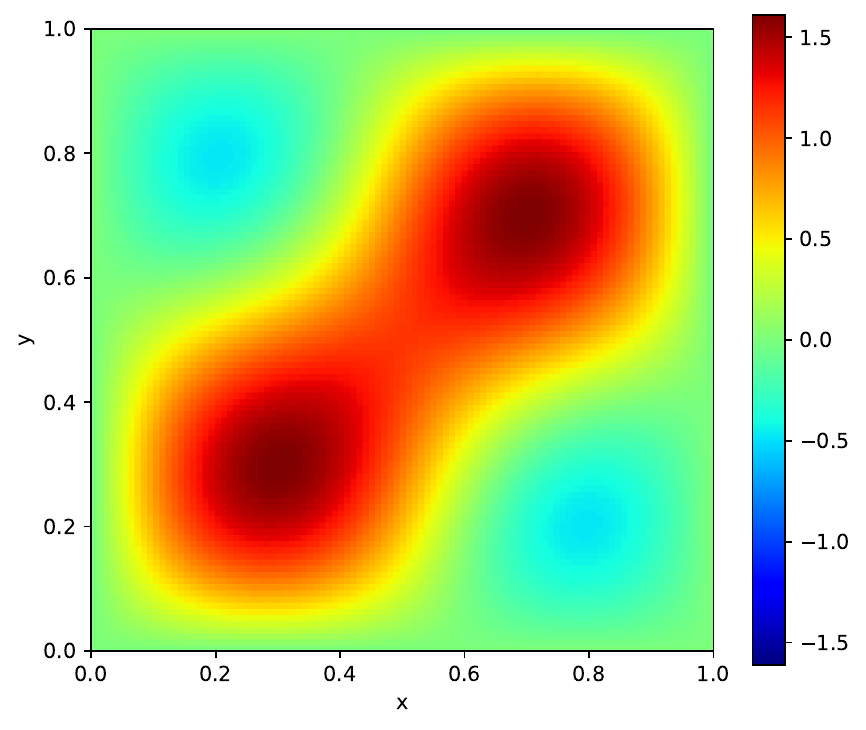}
        \caption{$\Re(u_{T})$}
    \end{subfigure}
    \hfill
    \begin{subfigure}[b]{0.32\textwidth}
        \centering
        \includegraphics[width=\textwidth]{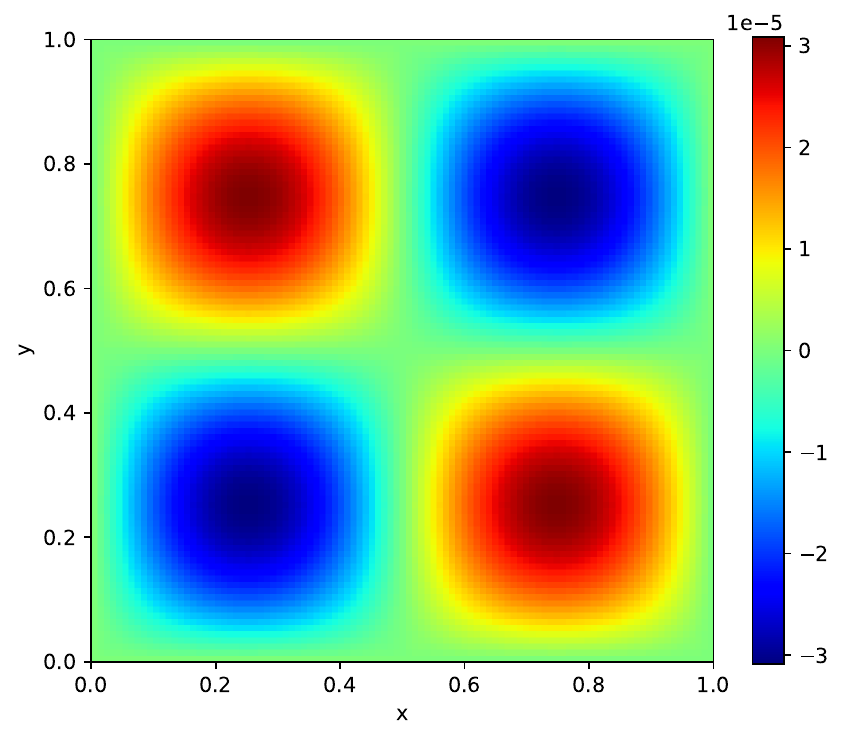}
        \caption{$\Re\big(\Psi(f)-u_{T}\big)$}
    \end{subfigure}
    \caption{Final-time comparison: real part. Parameters \((N_x,N_y,N_t)=(100,100,70)\), \(\\tau = 10^{-5}\), \(\epsilon=10^{-5}\).}
    \label{fig1}
\end{figure}

\begin{figure}[htbp]
    \centering
    \begin{subfigure}[b]{0.32\textwidth}
        \centering
        \includegraphics[width=\textwidth]{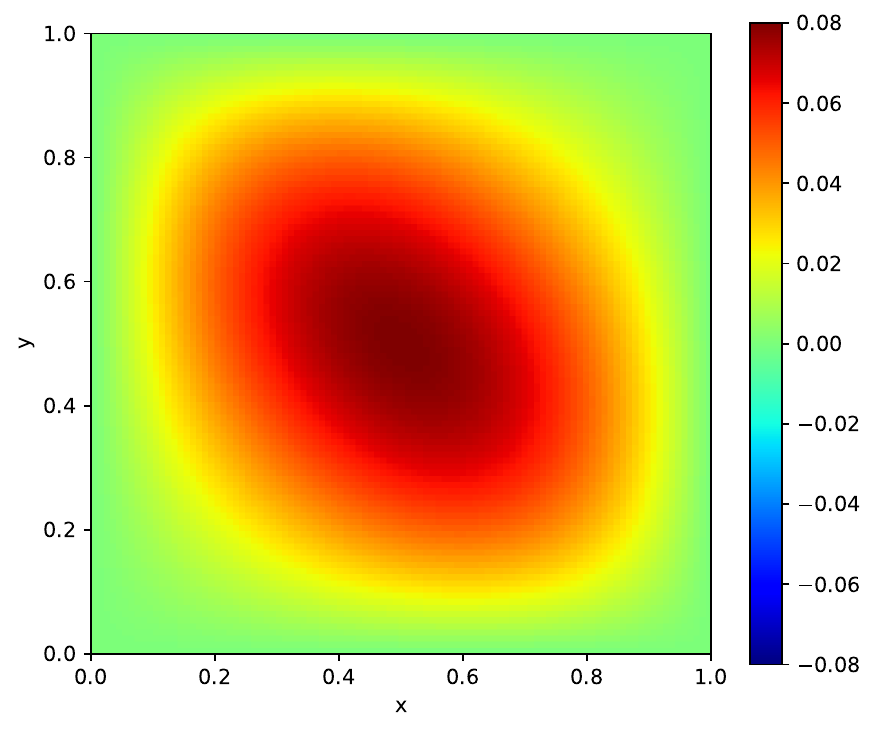}
        \caption{$\Im\big(\Psi(f)\big)$}
    \end{subfigure}
    \hfill
    \begin{subfigure}[b]{0.32\textwidth}
        \centering
        \includegraphics[width=\textwidth]{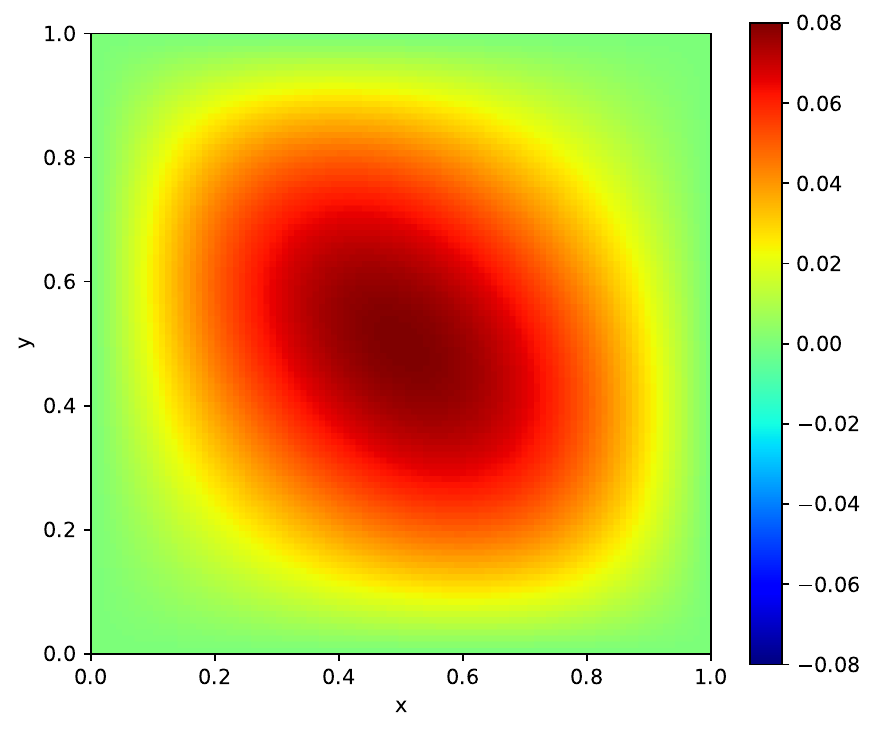}
        \caption{$\Im(u_{T})$}
    \end{subfigure}
    \hfill
    \begin{subfigure}[b]{0.32\textwidth}
        \centering
        \includegraphics[width=\textwidth]{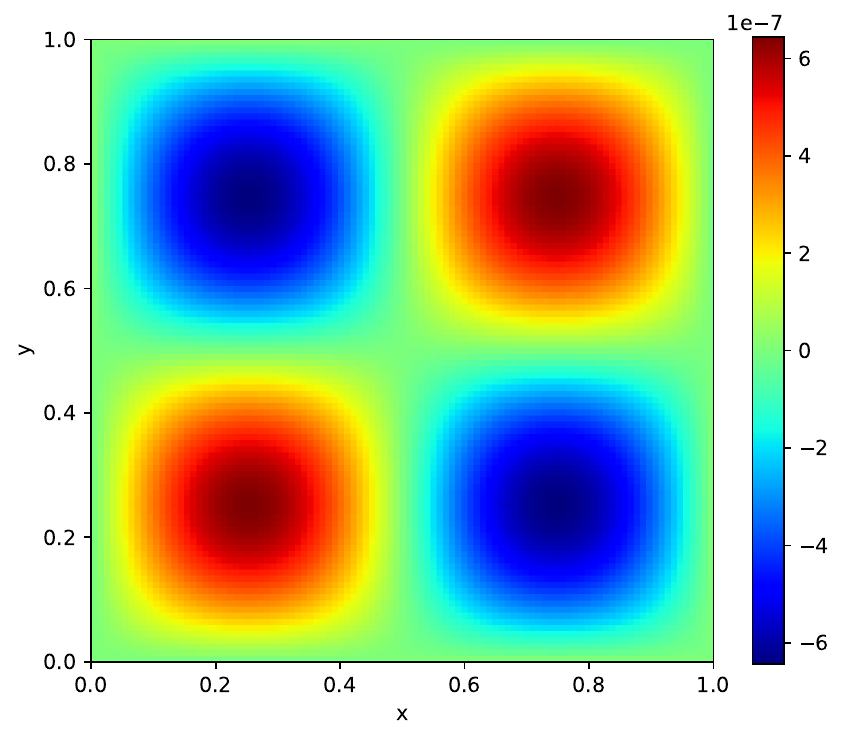}
        \caption{$\Im\big(\Psi(f)-u_{T}\big)$}
    \end{subfigure}
    \caption{Final-time comparison: imaginary part. Parameters \((N_x,N_y,N_t)=(100,100,70)\), \(\\tau = 10^{-5}\), \(\epsilon=10^{-5}\).}
    \label{fig2}
\end{figure}

Figures~\ref{fig3}--\ref{fig4} display the recovered source \(q_{\mathrm{rec}}\) and the true source \(q\). The reconstruction closely matches the ground truth; the residual plots confirm that the remaining discrepancies are small and spatially localized.

\begin{figure}[htbp]
    \centering
    \begin{subfigure}[b]{0.32\textwidth}
        \centering
        \includegraphics[width=\textwidth]{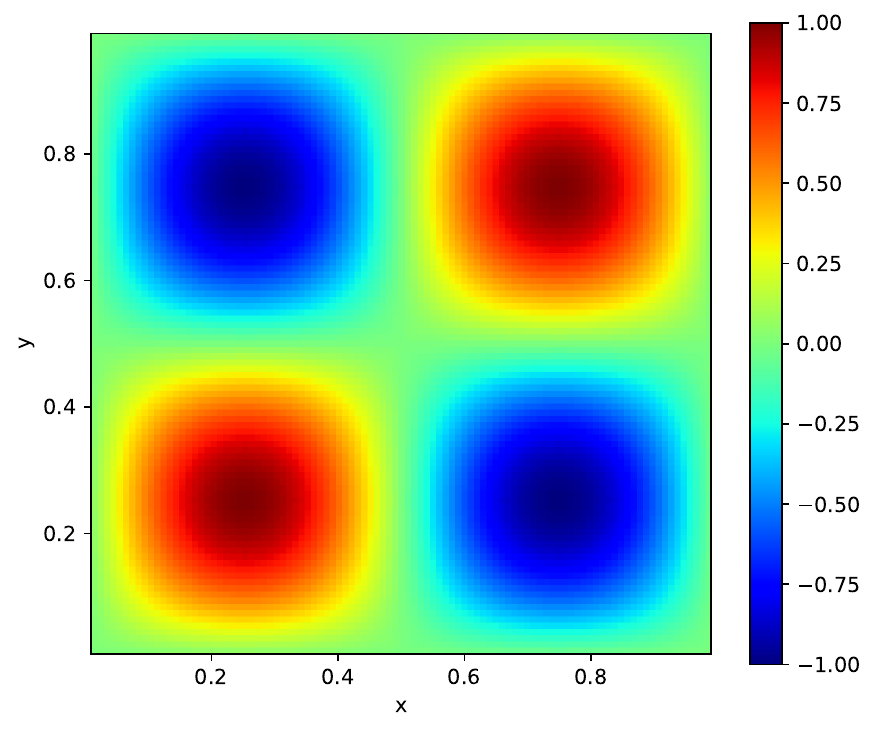}
        \caption{$\Re\big(q_{\mathrm{rec}}\big)$}
    \end{subfigure}
    \hfill
    \begin{subfigure}[b]{0.32\textwidth}
        \centering
        \includegraphics[width=\textwidth]{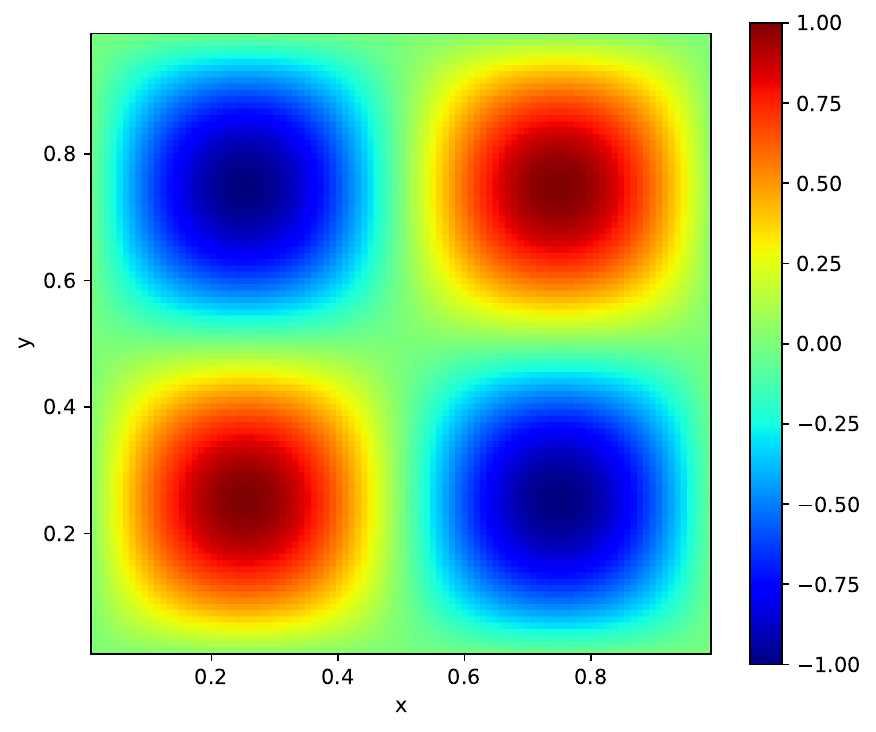}
        \caption{$\Re(q)$}
    \end{subfigure}
    \hfill
    \begin{subfigure}[b]{0.32\textwidth}
        \centering
        \includegraphics[width=\textwidth]{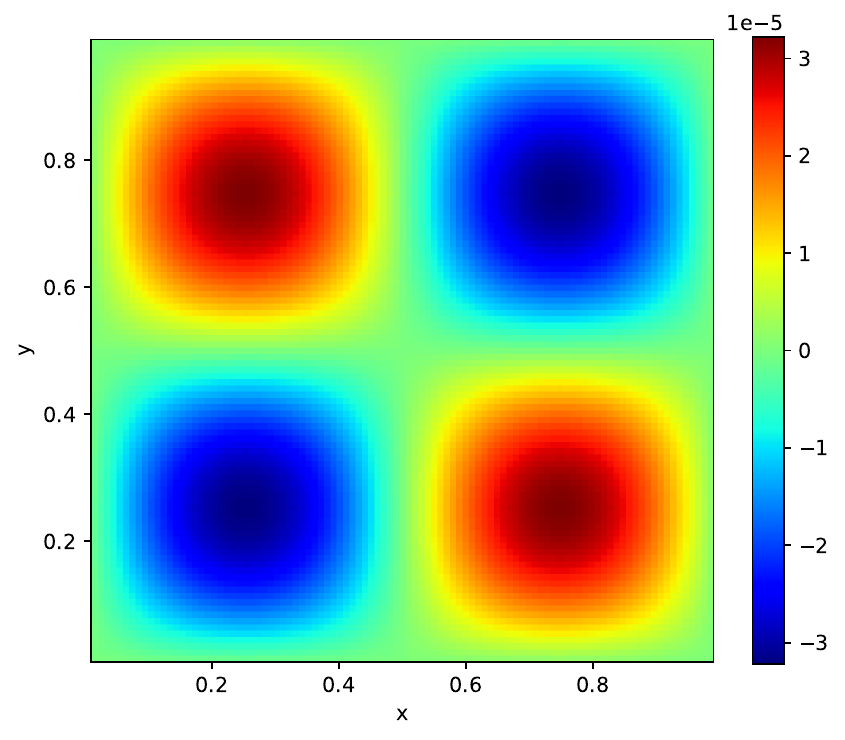}
        \caption{$\Re\big(q_{\mathrm{rec}}-q\big)$}
    \end{subfigure}
    \caption{Recovered vs.\ true source (real part). Parameters \((N_x,N_y,N_t)=(100,100,70)\), \(\\tau = 10^{-5}\), \(\epsilon=10^{-5}\).}
    \label{fig3}
\end{figure}

\begin{figure}[htbp]
    \centering
    \begin{subfigure}[b]{0.32\textwidth}
        \centering
        \includegraphics[width=\textwidth]{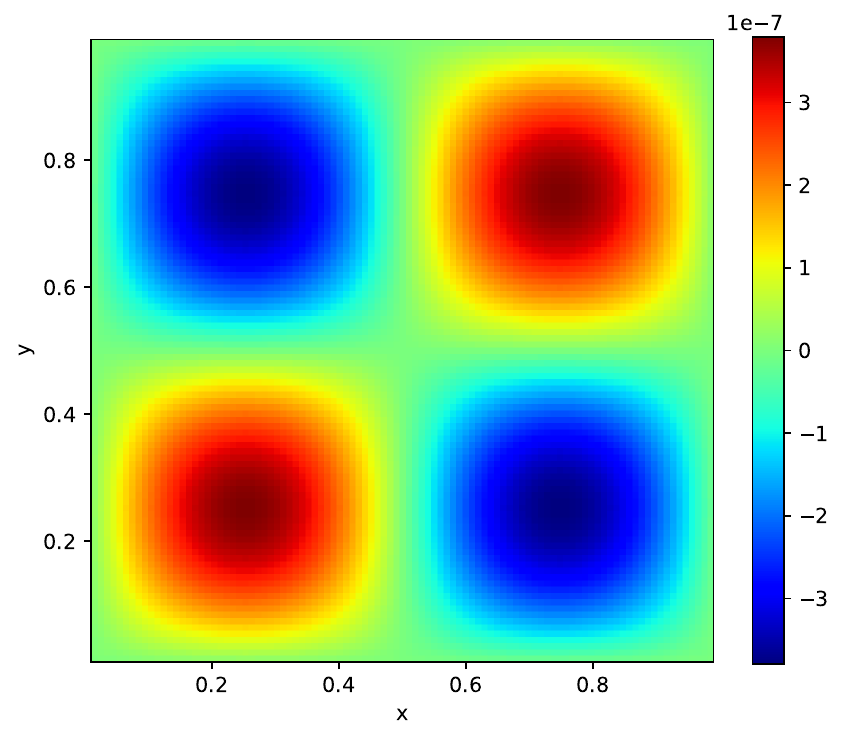}
        \caption{$\Im\big(q_{\mathrm{rec}}\big)$}
    \end{subfigure}
    \hfill
    \begin{subfigure}[b]{0.32\textwidth}
        \centering
        \includegraphics[width=\textwidth]{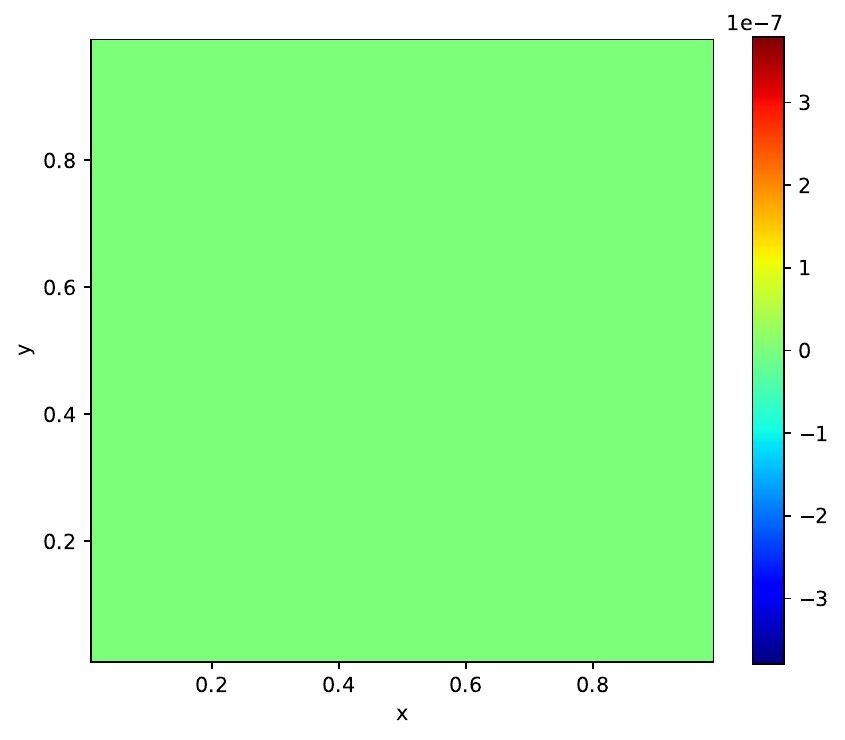}
        \caption{$\Im(q)$}
    \end{subfigure}
    \hfill
    \begin{subfigure}[b]{0.32\textwidth}
        \centering
        \includegraphics[width=\textwidth]{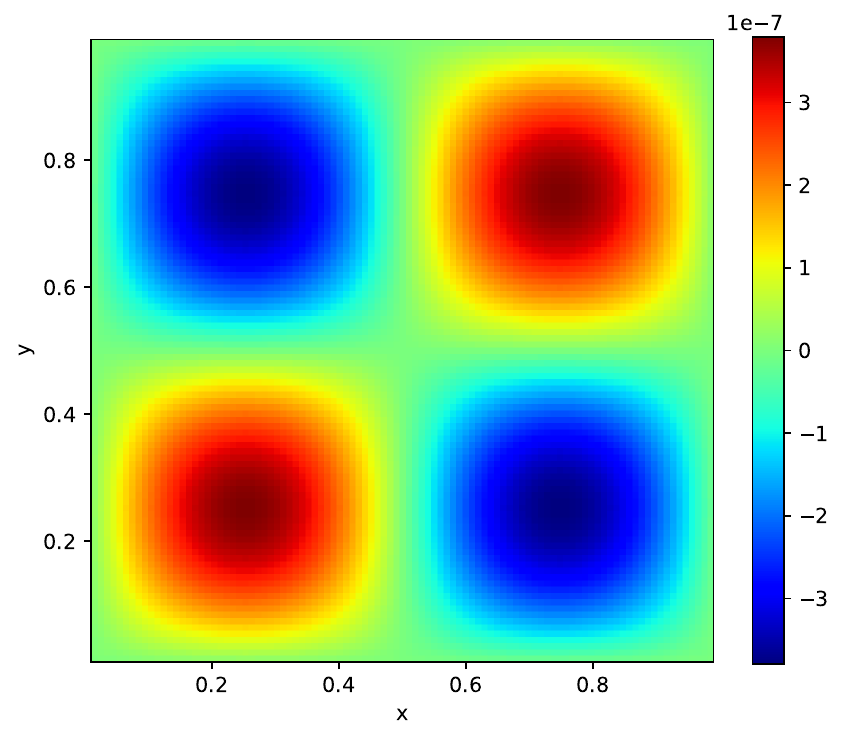}
        \caption{$\Im\big(q_{\mathrm{rec}}-q\big)$}
    \end{subfigure}
    \caption{Recovered vs.\ true source (imaginary part). Parameters \((N_x,N_y,N_t)=(100,100,70)\), \(\\tau = 10^{-5}\), \(\epsilon=10^{-5}\).}
    \label{fig4}
\end{figure}

\subsubsection{Example 2 (Localized imaginary Gaussian source)}

We next examine a localized imaginary source to test spatial selectivity and robustness of the inversion.

In this example we consider a purely imaginary, Gaussian–modulated source,
\begin{equation}\label{eq:ex2q}
    q(x,y)= i\,\exp\!\left(-\frac{(x-\tfrac12)^2+(y-\tfrac12)^2}{2\sigma^2}\right)
    \sin(\pi x)\,\sin(\pi y),\qquad \sigma=0.12,
\end{equation}
which is spatially localized around the domain center.

Figures~\ref{fig5}--\ref{fig6} compare the reconstructed and measured final states, showing the real and imaginary parts of \(\Psi(f)\) and \(u_T\). The plots indicate excellent agreement; the relative (discrete) \(L^2\) data misfit is
\[
\frac{\Vert \Psi(f)-u_T\Vert_{h}^{2}}{\Vert u_T\Vert_{h}^{2}} \;=\; 4.2617\times 10^{-5}.
\]

\begin{figure}[htbp]
    \centering
    \begin{subfigure}[b]{0.32\textwidth}
        \centering
        \includegraphics[width=\textwidth]{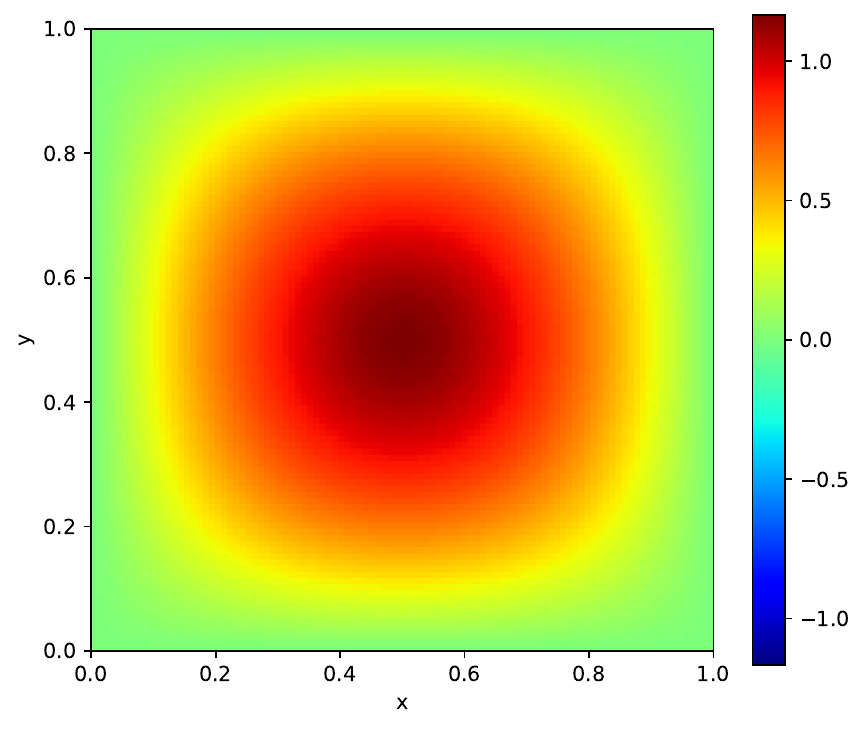}
        \caption{$\Re(y(\cdot,T))$}
    \end{subfigure}
    \hfill
    \begin{subfigure}[b]{0.32\textwidth}
        \centering
        \includegraphics[width=\textwidth]{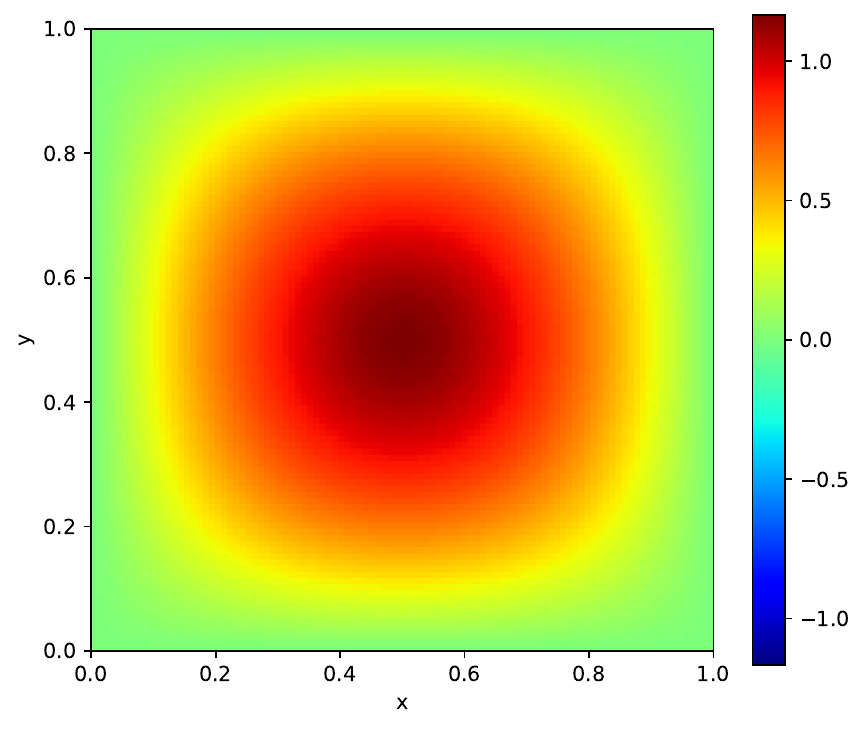}
        \caption{$\Re(u_{T})$}
    \end{subfigure}
    \hfill
    \begin{subfigure}[b]{0.32\textwidth}
        \centering
        \includegraphics[width=\textwidth]{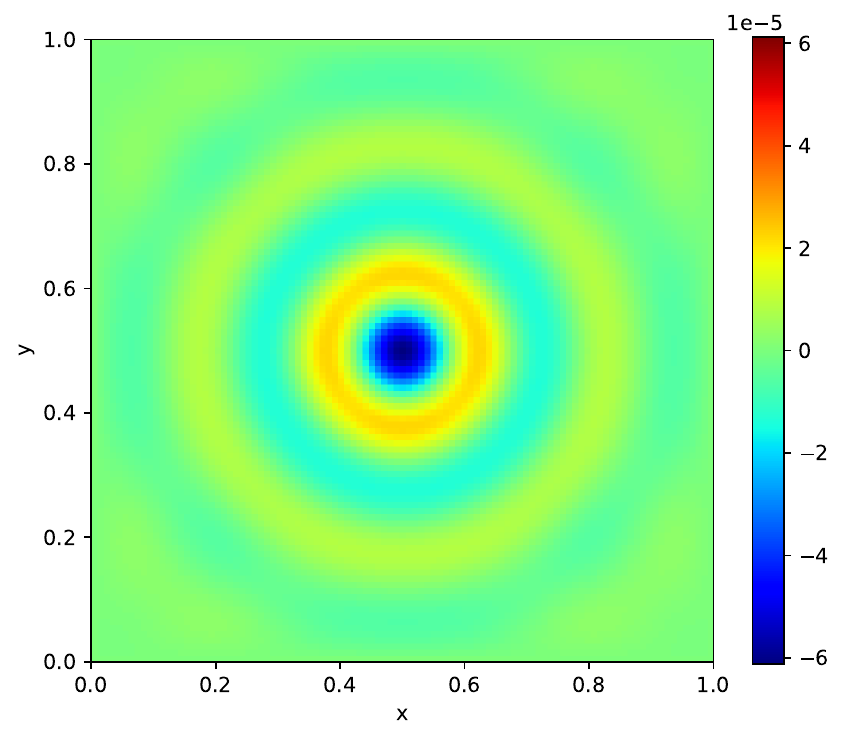}
        \caption{$\Re(y(\cdot,T)-u_{T})$}
    \end{subfigure}
    \caption{Final-time comparison: real part. Parameters \((N_x,N_y,N_t)=(100,100,70)\), \(\tau = 10^{-5}\), \(\varepsilon=10^{-5}\).}
    \label{fig5}
\end{figure}

\begin{figure}[htbp]
    \centering
    \begin{subfigure}[b]{0.32\textwidth}
        \centering
        \includegraphics[width=\textwidth]{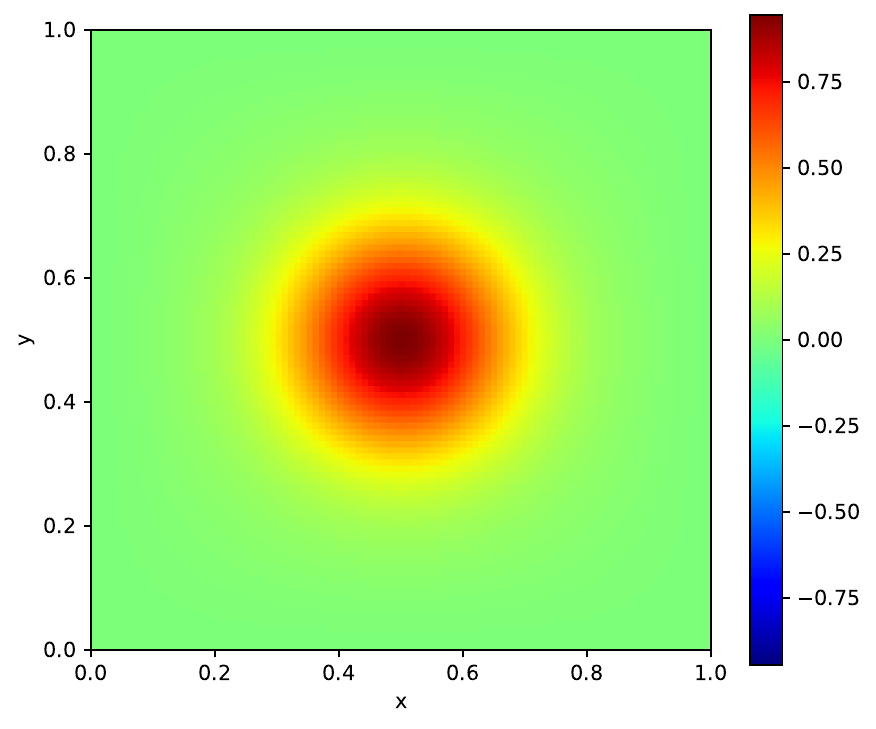}
        \caption{$\Im(y(\cdot,T))$}
    \end{subfigure}
    \hfill
    \begin{subfigure}[b]{0.32\textwidth}
        \centering
        \includegraphics[width=\textwidth]{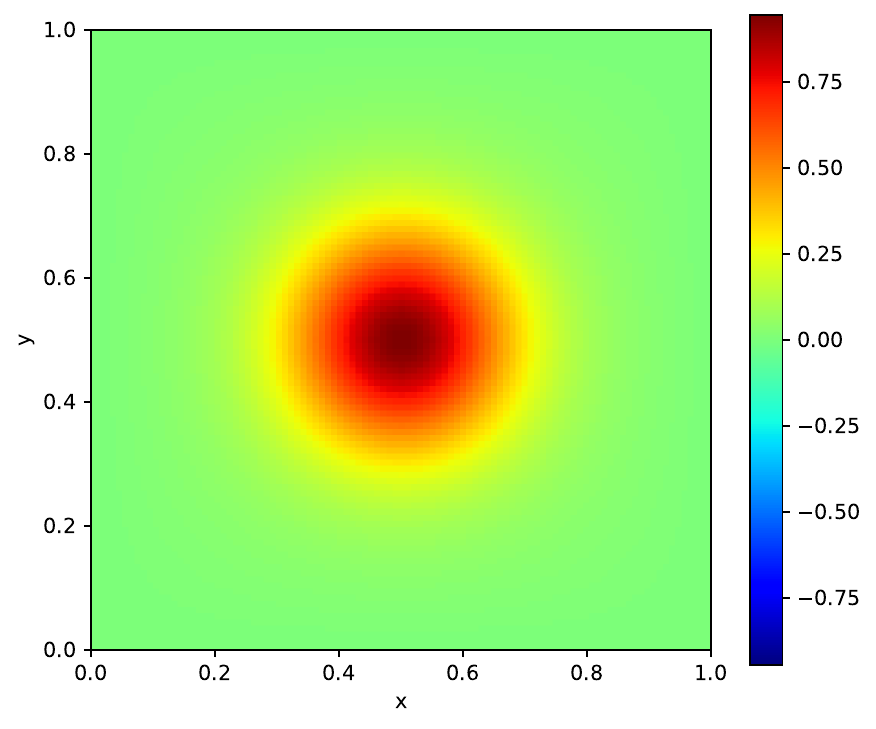}
        \caption{$\Im(u_{T})$}
    \end{subfigure}
    \hfill
    \begin{subfigure}[b]{0.32\textwidth}
        \centering
        \includegraphics[width=\textwidth]{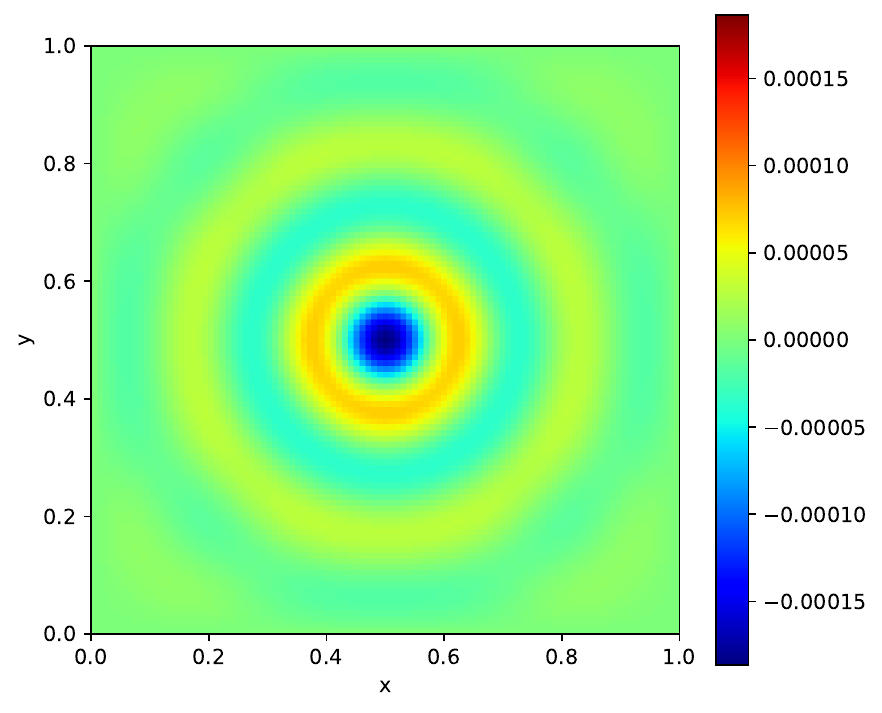}
        \caption{$\Im(y(\cdot,T)-u_{T})$}
    \end{subfigure}
    \caption{Final-time comparison: imaginary part. Parameters \((N_x,N_y,N_t)=(100,100,70)\), \(\tau = 10^{-5}\), \(\varepsilon=10^{-5}\).}
    \label{fig6}
\end{figure}

Figures~\ref{fig7}--\ref{fig8} display the recovered source \(q_{\mathrm{rec}}\) alongside the ground truth \(q\) from \eqref{eq:ex2q}. The reconstruction is highly accurate; the residuals are small and spatially localized. The relative (discrete) \(L^2\) reconstruction error is
\[
\frac{\Vert q_{\mathrm{rec}}-q\Vert_{h}^{2}}{\Vert q \Vert_{h}^{2}}
= 4.3644\times 10^{-4}.
\]

\begin{figure}[htbp]
    \centering
    \begin{subfigure}[b]{0.32\textwidth}
        \centering
        \includegraphics[width=\textwidth]{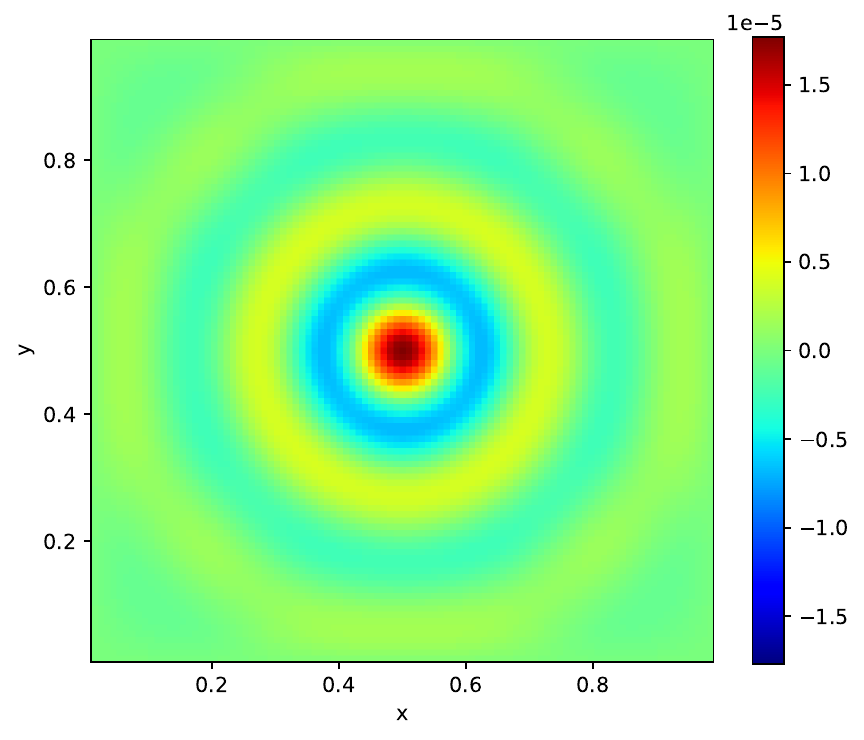}
        \caption{$\Re(q_{rec})$}
    \end{subfigure}
    \hfill
    \begin{subfigure}[b]{0.32\textwidth}
        \centering
        \includegraphics[width=\textwidth]{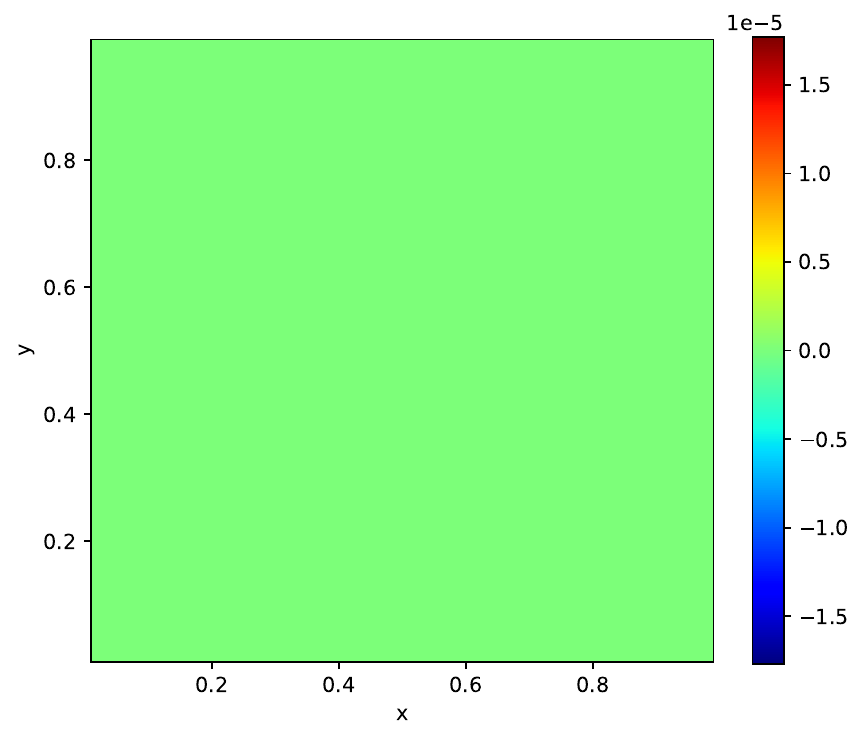}
        \caption{$\Re(q)$}
    \end{subfigure}
    \hfill
    \begin{subfigure}[b]{0.32\textwidth}
        \centering
        \includegraphics[width=\textwidth]{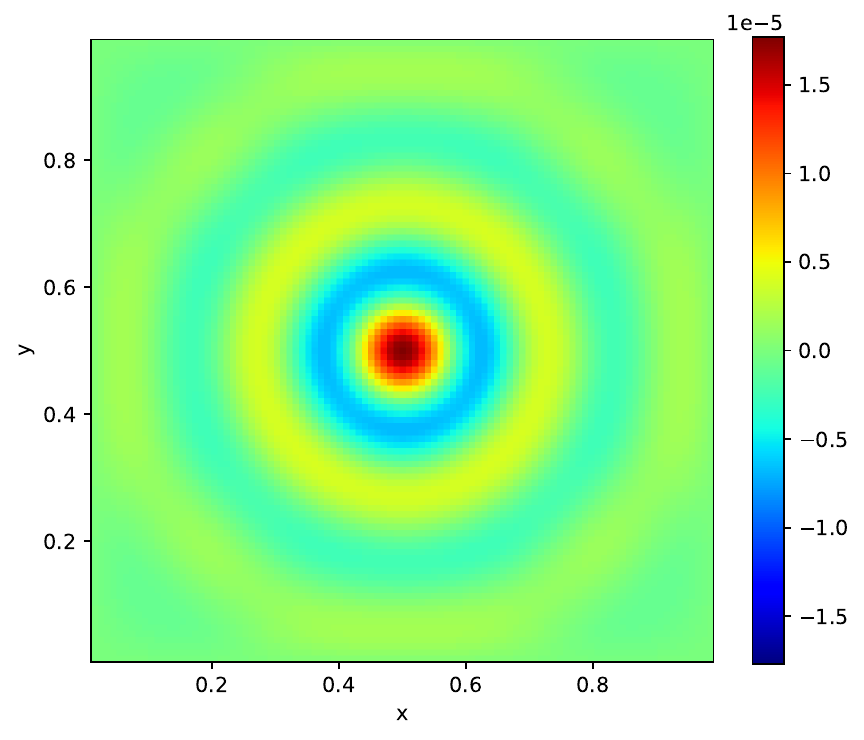}
        \caption{$\Re(q^{rec}-q)$}
    \end{subfigure}
    \caption{Recovered vs.\ true source (real part). Parameters \((N_x,N_y,N_t)=(100,100,70)\), \(\tau = 10^{-5}\), \(\varepsilon=10^{-5}\).}
    \label{fig7}
\end{figure}

\begin{figure}[htbp]
    \centering
    \begin{subfigure}[b]{0.32\textwidth}
        \centering
        \includegraphics[width=\textwidth]{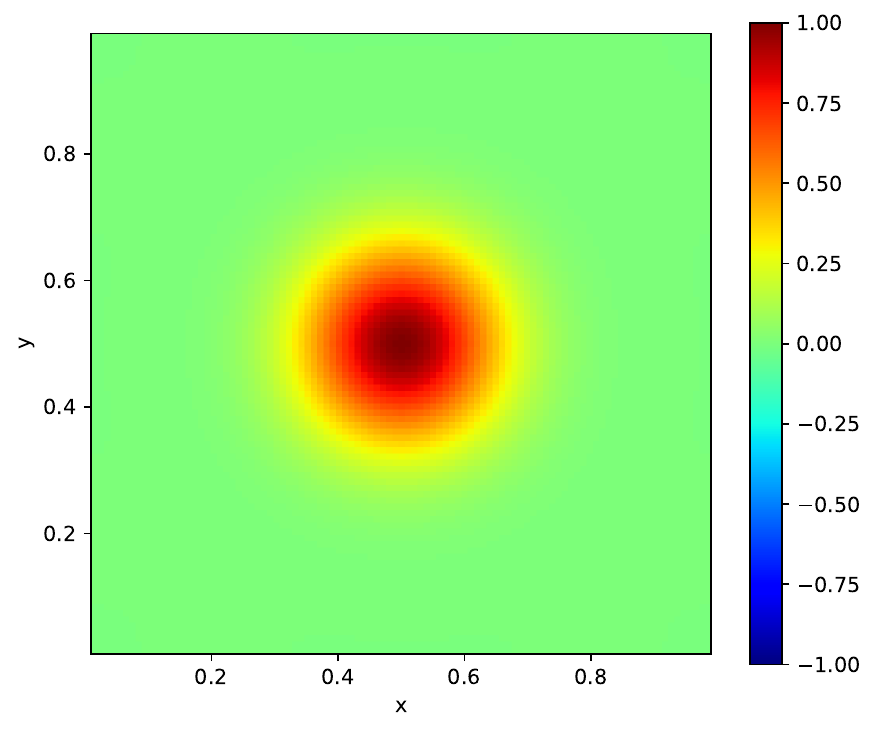}
        \caption{$\Im(q_{rec})$}
    \end{subfigure}
    \hfill
    \begin{subfigure}[b]{0.32\textwidth}
        \centering
        \includegraphics[width=\textwidth]{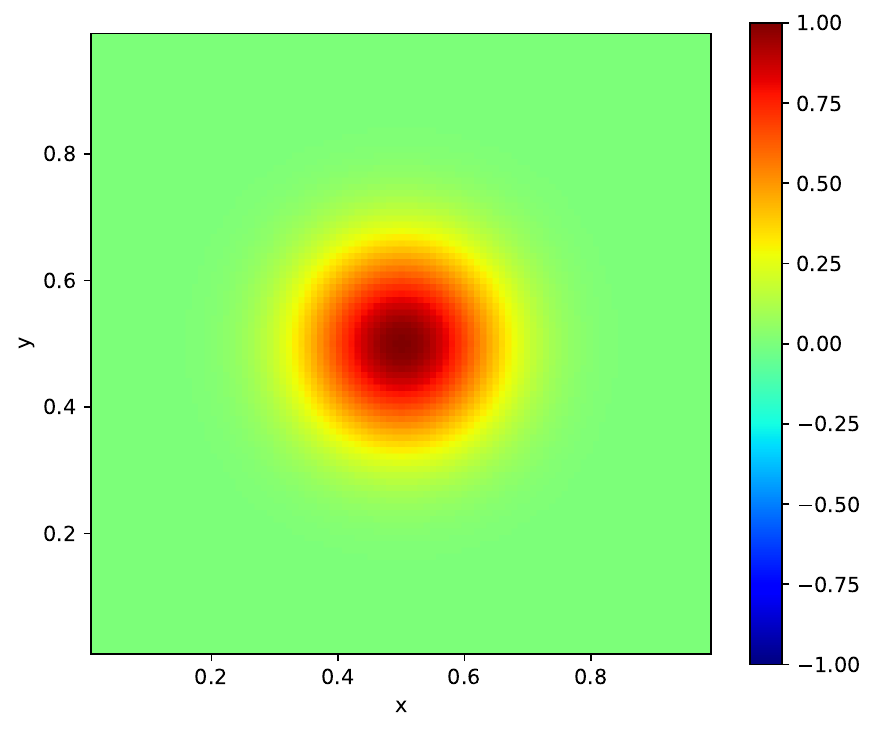}
        \caption{$\Im(q)$}
    \end{subfigure}
    \hfill
    \begin{subfigure}[b]{0.32\textwidth}
        \centering
        \includegraphics[width=\textwidth]{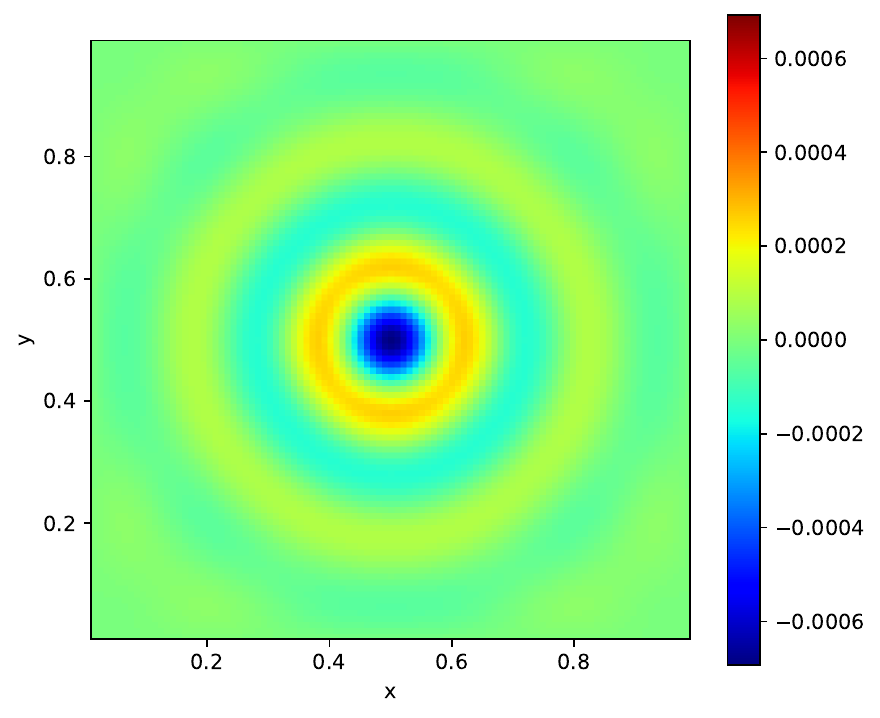}
        \caption{$\Im(q^{rec}-q)$}
    \end{subfigure}
    \caption{Recovered vs.\ true source (imaginary part). Parameters \((N_x,N_y,N_t)=(100,100,70)\), \(\tau = 10^{-5}\), \(\varepsilon=10^{-5}\).}
    \label{fig8}
\end{figure}

\subsubsection{Example 3 (Complex source with mixed frequencies)}

This example explores a complex source that combines distinct spatial frequencies in its real and imaginary parts. We consider
\begin{equation}\label{eq:ex3q}
    q(x,y)= \sin(2\pi x)\,\sin(2\pi y)\;+\; i\,0.7\,\sin(3\pi x)\,\sin(2\pi y),
    \qquad (x,y)\in(0,1)^2.
\end{equation}

Figures~\ref{fig9}--\ref{fig10} compare the reconstructed and measured final states, displaying the real and imaginary parts of \(\Psi(f)\) and \(u_T\). The agreement is excellent; the relative (discrete) \(L^2\) data misfit equals
\[
\frac{\Vert \Psi(f)-u_T\Vert_{h}^{2}}{\Vert u_T\Vert_{h}^{2}} \;=\; 1.9803\times 10^{-5}.
\]

\begin{figure}[htbp]
    \centering
    \begin{subfigure}[b]{0.32\textwidth}
        \centering
        \includegraphics[width=\textwidth]{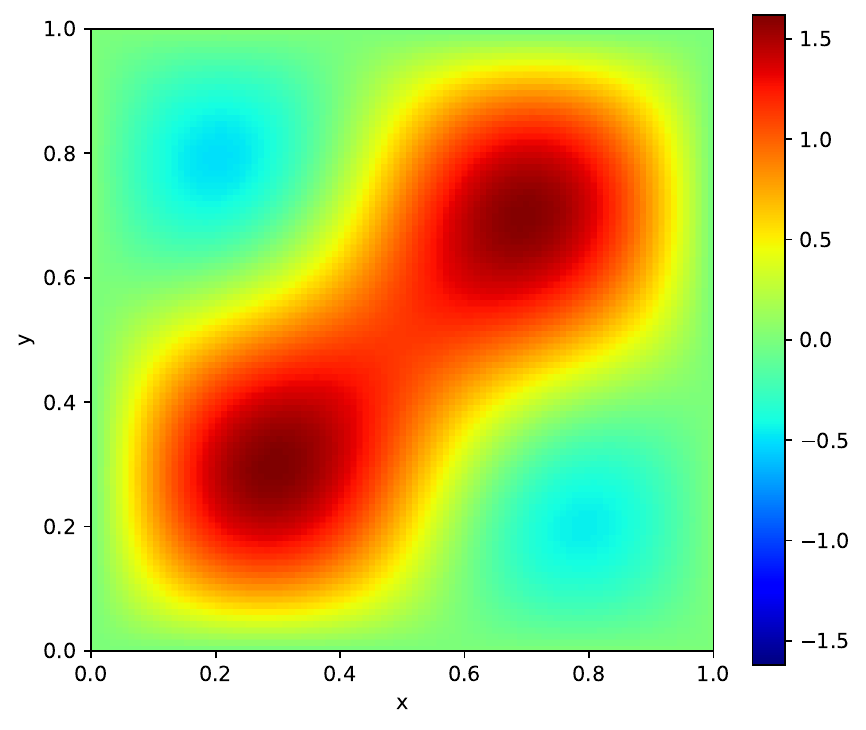}
        \caption{$\Re\big(\Psi(f)\big)$}
    \end{subfigure}\hfill
    \begin{subfigure}[b]{0.32\textwidth}
        \centering
        \includegraphics[width=\textwidth]{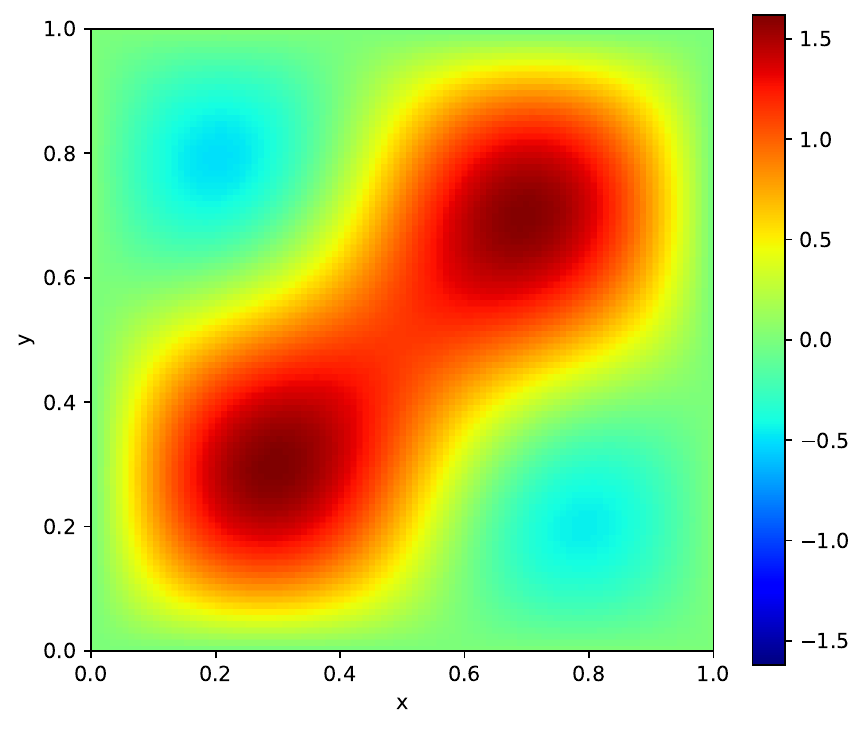}
        \caption{$\Re(u_{T})$}
    \end{subfigure}\hfill
    \begin{subfigure}[b]{0.32\textwidth}
        \centering
        \includegraphics[width=\textwidth]{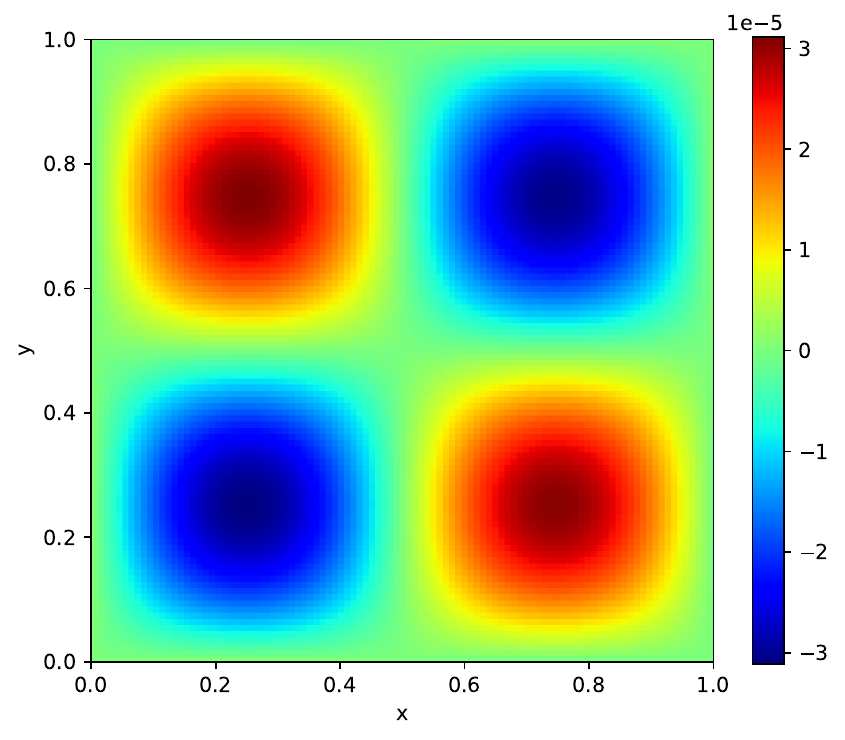}
        \caption{$\Re\big(\Psi(f)-u_{T}\big)$}
    \end{subfigure}
    \caption{Final-time comparison: real part. Parameters \((N_x,N_y,N_t)=(100,100,70)\), \(\tau = 10^{-5}\), \(\varepsilon=10^{-5}\).}
    \label{fig9}
\end{figure}

\begin{figure}[htbp]
    \centering
    \begin{subfigure}[b]{0.32\textwidth}
        \centering
        \includegraphics[width=\textwidth]{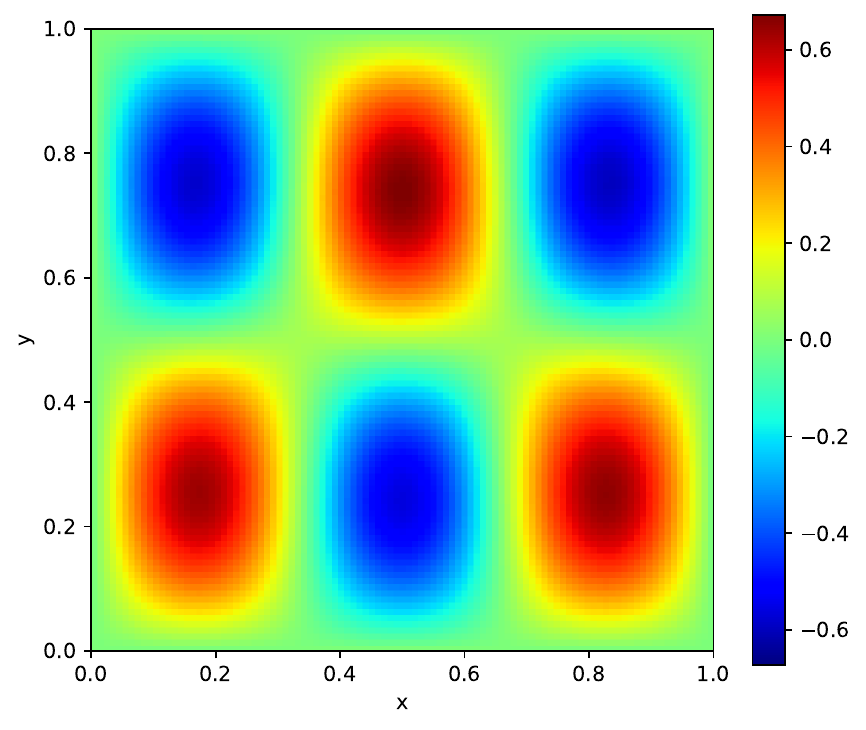}
        \caption{$\Im\big(\Psi(f)\big)$}
    \end{subfigure}\hfill
    \begin{subfigure}[b]{0.32\textwidth}
        \centering
        \includegraphics[width=\textwidth]{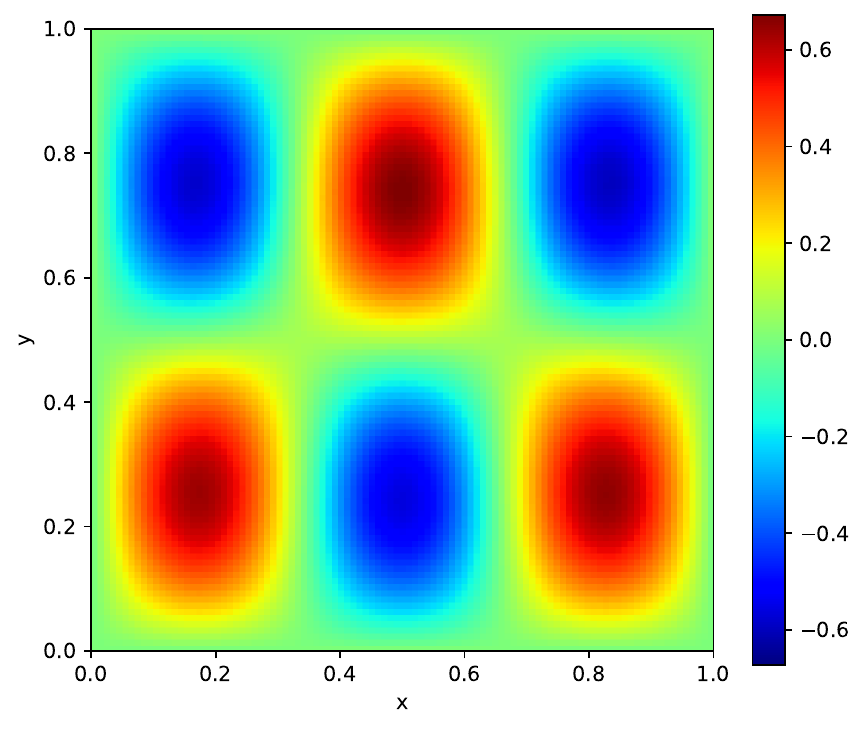}
        \caption{$\Im(u_{T})$}
    \end{subfigure}\hfill
    \begin{subfigure}[b]{0.32\textwidth}
        \centering
        \includegraphics[width=\textwidth]{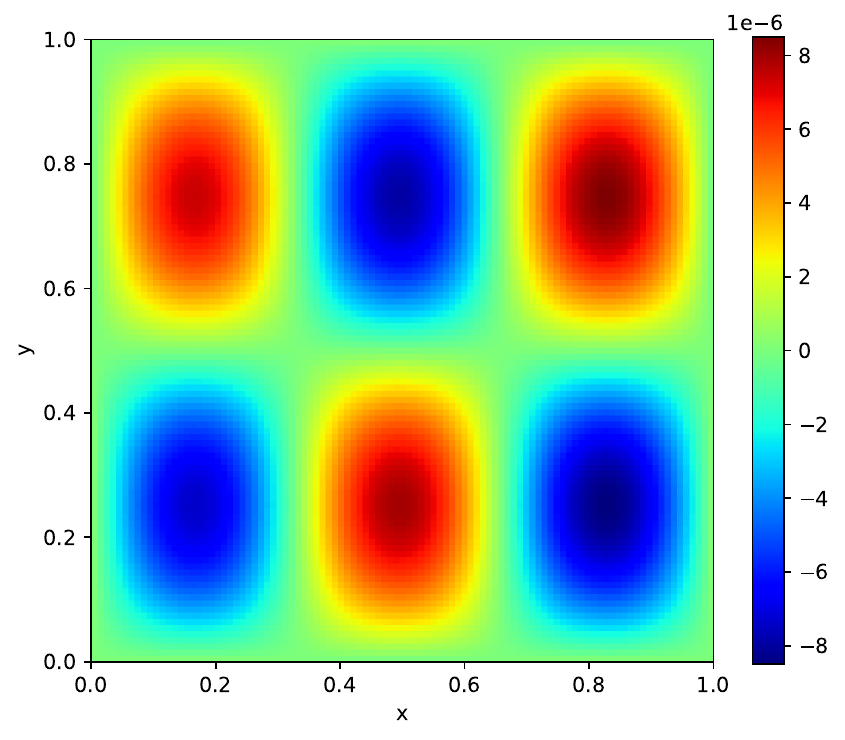}
        \caption{$\Im\big(\Psi(f)-u_{T}\big)$}
    \end{subfigure}
    \caption{Final-time comparison: imaginary part. Parameters \((N_x,N_y,N_t)=(100,100,70)\), \(\tau = 10^{-5}\), \(\varepsilon=10^{-5}\).}
    \label{fig10}
\end{figure}

Figures~\ref{fig11}--\ref{fig12} display the recovered source \(q_{\mathrm{rec}}\) alongside the ground truth \(q\) from \eqref{eq:ex3q}. The reconstruction is highly accurate; residuals are small and spatially localized. The relative (discrete) \(L^2\) reconstruction error is
\[
\frac{\Vert q_{\mathrm{rec}}-q\Vert_{h}^{2}}{\Vert q \Vert_{h}^{2}}
= 2.7399\times 10^{-5}.
\]

\begin{figure}[htbp]
    \centering
    \begin{subfigure}[b]{0.32\textwidth}
        \centering
        \includegraphics[width=\textwidth]{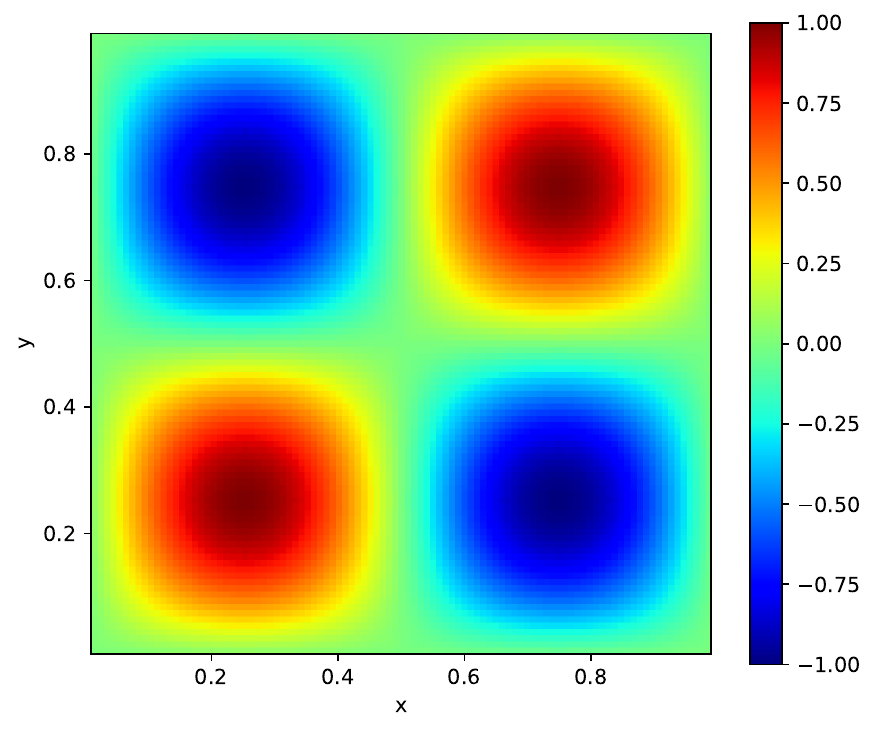}
        \caption{$\Re\big(q_{\mathrm{rec}}\big)$}
    \end{subfigure}\hfill
    \begin{subfigure}[b]{0.32\textwidth}
        \centering
        \includegraphics[width=\textwidth]{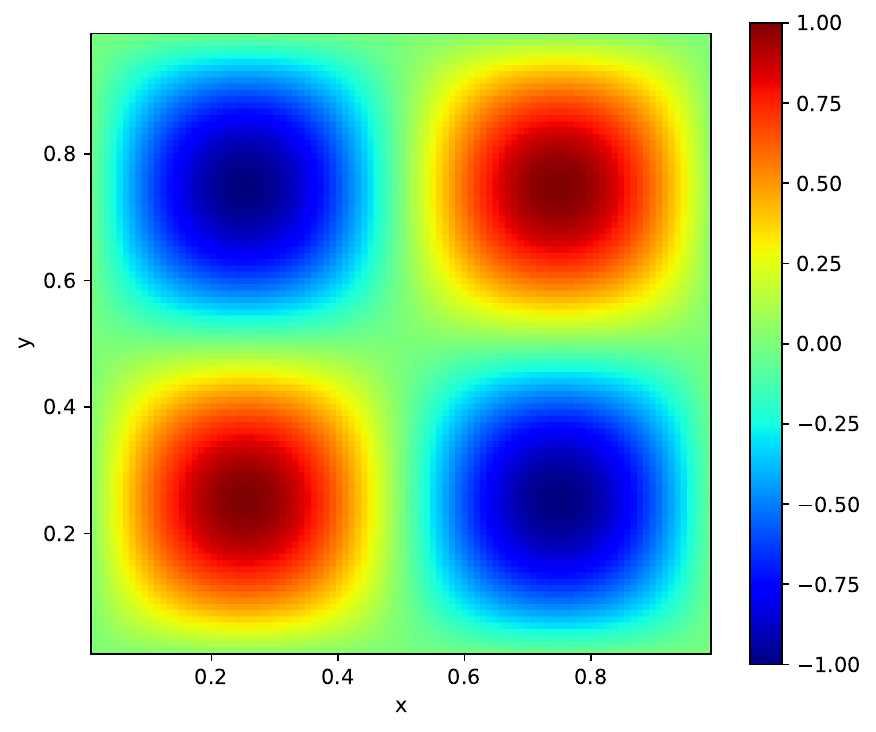}
        \caption{$\Re(q)$}
    \end{subfigure}\hfill
    \begin{subfigure}[b]{0.32\textwidth}
        \centering
        \includegraphics[width=\textwidth]{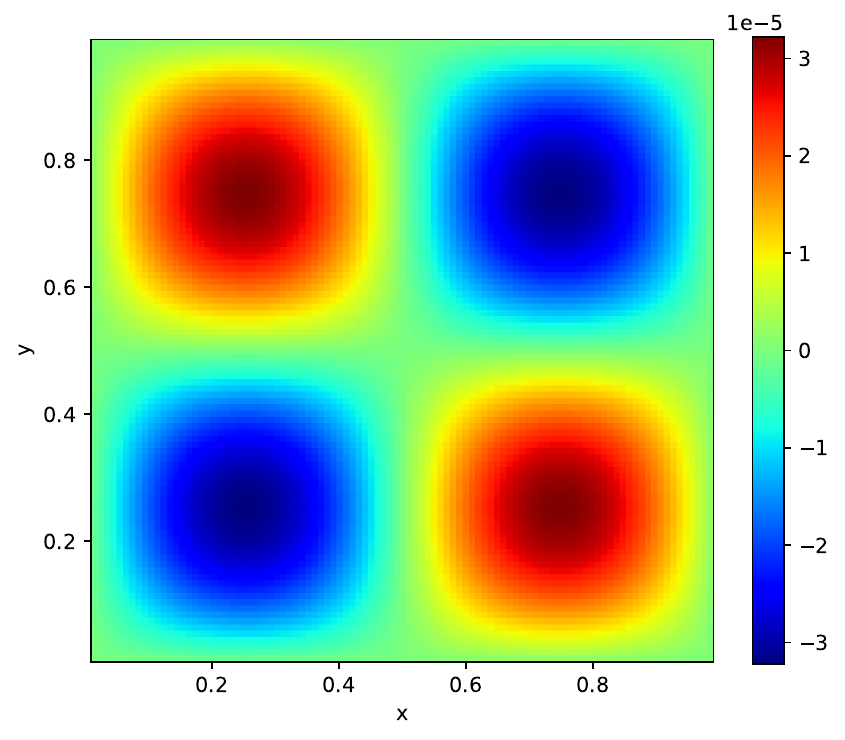}
        \caption{$\Re\big(q_{\mathrm{rec}}-q\big)$}
    \end{subfigure}
    \caption{Recovered vs.\ true source (real part). Parameters \((N_x,N_y,N_t)=(100,100,70)\), \(\tau = 10^{-5}\), \(\varepsilon=10^{-5}\).}
    \label{fig11}
\end{figure}

\begin{figure}[htbp]
    \centering
    \begin{subfigure}[b]{0.32\textwidth}
        \centering
        \includegraphics[width=\textwidth]{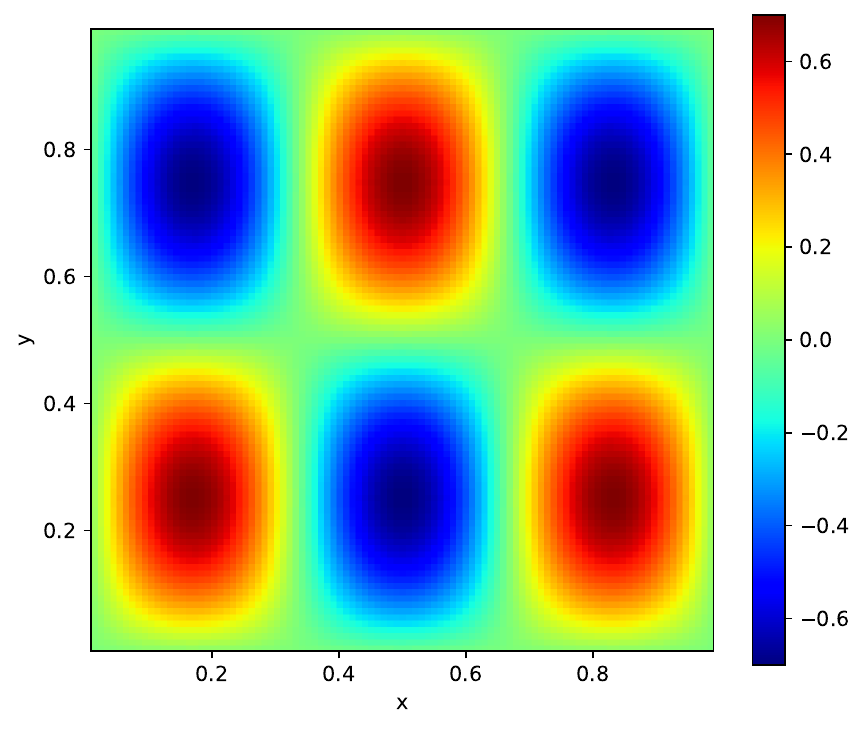}
        \caption{$\Im\big(q_{\mathrm{rec}}\big)$}
    \end{subfigure}\hfill
    \begin{subfigure}[b]{0.32\textwidth}
        \centering
        \includegraphics[width=\textwidth]{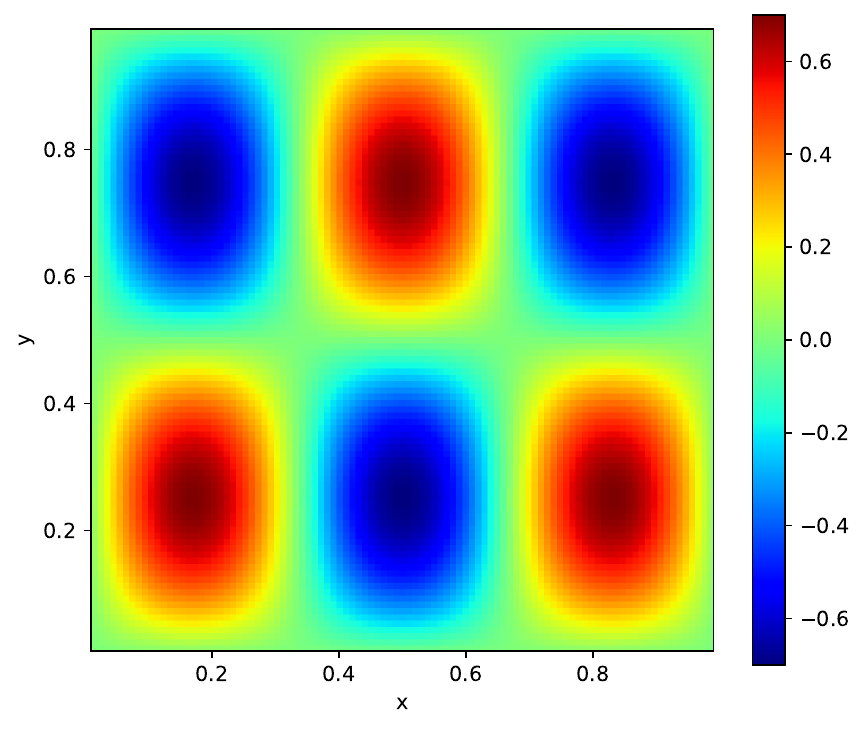}
        \caption{$\Im(q)$}
    \end{subfigure}\hfill
    \begin{subfigure}[b]{0.32\textwidth}
        \centering
        \includegraphics[width=\textwidth]{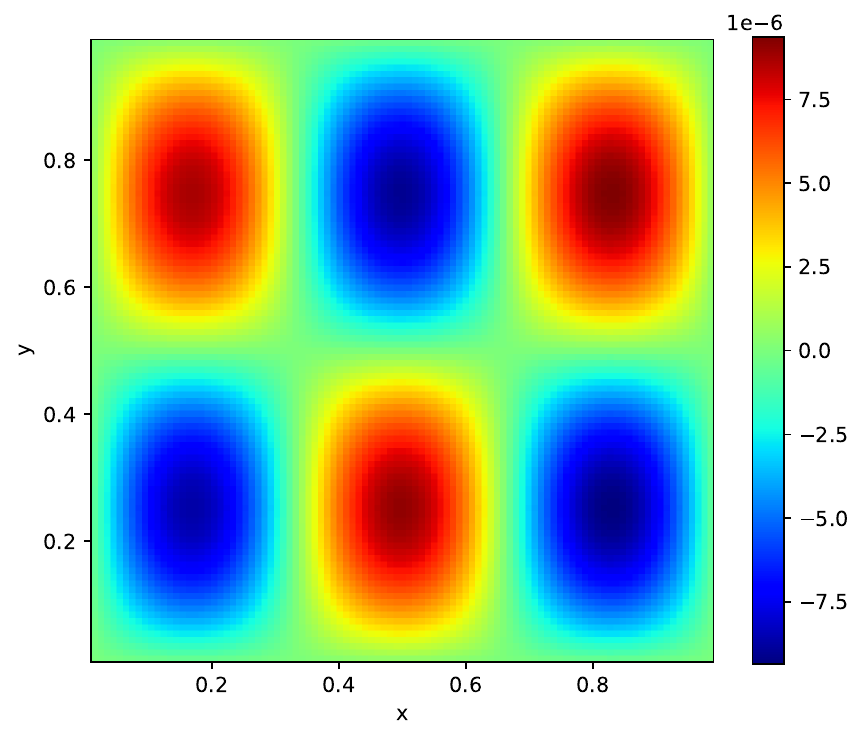}
        \caption{$\Im\big(q_{\mathrm{rec}}-q\big)$}
    \end{subfigure}
    \caption{Recovered vs.\ true source (imaginary part). Parameters \((N_x,N_y,N_t)=(100,100,70)\), \(\tau = 10^{-5}\), \(\varepsilon=10^{-5}\).}
    \label{fig12}
\end{figure}

To assess robustness, we corrupt the exact final state \(u_T=u(\cdot,T)\) with circularly symmetric complex Gaussian noise at a prescribed relative level \(\delta\in(0,1)\). Let \(\xi=\xi_1+i\,\xi_2\), where \(\xi_1,\xi_2\) have i.i.d.\ standard normal entries, and define
\begin{equation}\label{eq:noisy-data}
u_T^\delta \;=\; u_T \;+\; \eta^\delta,
\qquad
\eta^\delta \;=\; \delta\,\frac{\|u_T\|_h}{\|\xi\|_h}\,\xi,
\end{equation}
so that \(\|u_T^\delta-u_T\|_h/\|u_T\|_h=\delta\). In reconstructions with noise we replace the target \(v_h^{N_t}\) in~\eqref{eq:J-discrete} by \(u_T^\delta\); the algorithm is otherwise unchanged. Misfit values are reported in the same discrete norm \(\|\cdot\|_h\).

\begin{table}[htbp]
\caption{Noisy-data reconstructions for \(\varepsilon = 10^{-5}\) and \(\tau = 10^{-5}\). The misfit is computed with \(u_T^\delta\) in place of \(u_T\) (for \(\delta=0\), \(u_T^\delta=u_T\)).}
\label{tab:ex3:noise}
\centering
\begin{tabular}{l | l l l}
\toprule
 \(\delta\) & iterations & \(\dfrac{\Vert \Psi(f)-u_T^\delta\Vert_{h}^{2}}{\Vert u_T^\delta\Vert_{h}^{2}}\) & \(\dfrac{\Vert q_{\mathrm{rec}}-q\Vert_{h}^{2}}{\Vert q \Vert_{h}^{2}}\) \\
\midrule
\(0\)                & 62   & \(1.9803\times 10^{-5}\) & \(2.7399\times 10^{-5}\) \\
\(1\times 10^{-3}\)  & 443  & \(8.5539\times 10^{-4}\) & \(1.8720\times 10^{-2}\) \\
\(5\times 10^{-3}\)  & 8764 & \(1.4810\times 10^{-3}\) & \(7.3337\times 10^{-1}\) \\
\bottomrule
\end{tabular}
\end{table}

In the presence of noise, the Tikhonov parameter \(\varepsilon\) must balance bias and variance. For instance, for \(\delta=10^{-3}\) we also considered \(\varepsilon=10^{-3}\); in that case, after \(109\) iterations the relative source error is
\[
\frac{\Vert q_{\mathrm{rec}}-q\Vert_{h}^{2}}{\Vert q \Vert_{h}^{2}}
= 5.1450\times 10^{-3}.
\]

\subsubsection{Example 4 (Space–time forcing)}

In this example we consider a space–time forcing of the form
\begin{equation}\label{eq:ex4f}
    f(x,y,t)\;=\; i\,\exp\!\left(-\frac{\big((x-\tfrac12)^2+(y-\tfrac12)^2\big)\,t}{2\sigma^2}\right)
    \sin(\pi x)\,\sin(\pi y),\qquad \sigma=0.12.
\end{equation}

Figures~\ref{fig13}--\ref{fig14} compare the reconstructed and measured final states, showing the real and imaginary parts of \(\Psi(f)\) and \(u_T\). The agreement is excellent; the relative (discrete) \(L^2\) data misfit is
\[
\frac{\Vert \Psi(f)-u_T\Vert_{h}^{2}}{\Vert u_T\Vert_{h}^{2}} \;=\; 1.0473\times 10^{-5}.
\]

\begin{figure}[htbp]
    \centering
    \begin{subfigure}[b]{0.32\textwidth}
        \centering
        \includegraphics[width=\textwidth]{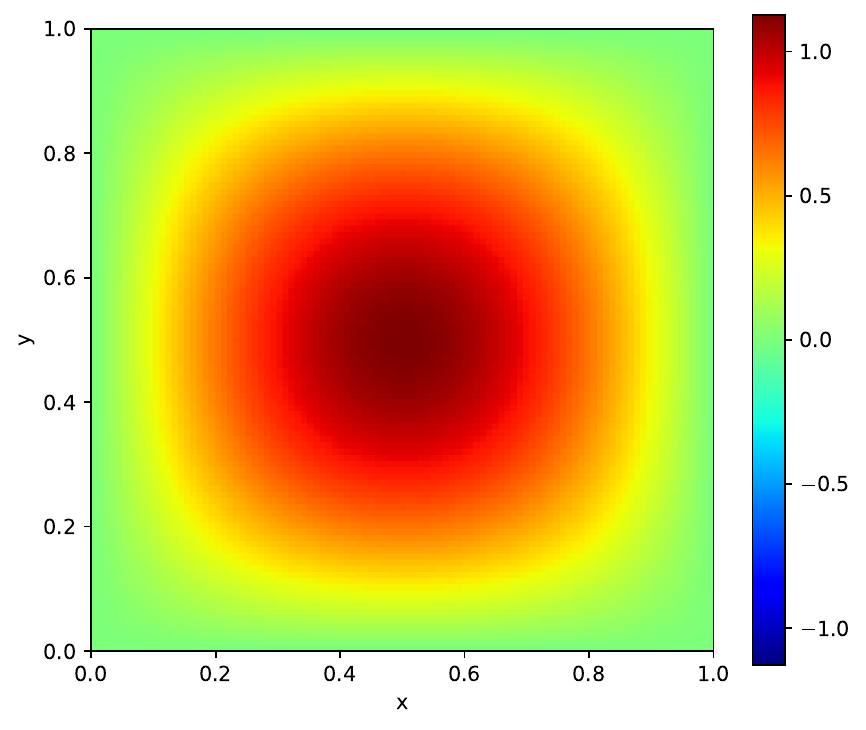}
        \caption{$\Re\big(\Psi(f)\big)$}
    \end{subfigure}\hfill
    \begin{subfigure}[b]{0.32\textwidth}
        \centering
        \includegraphics[width=\textwidth]{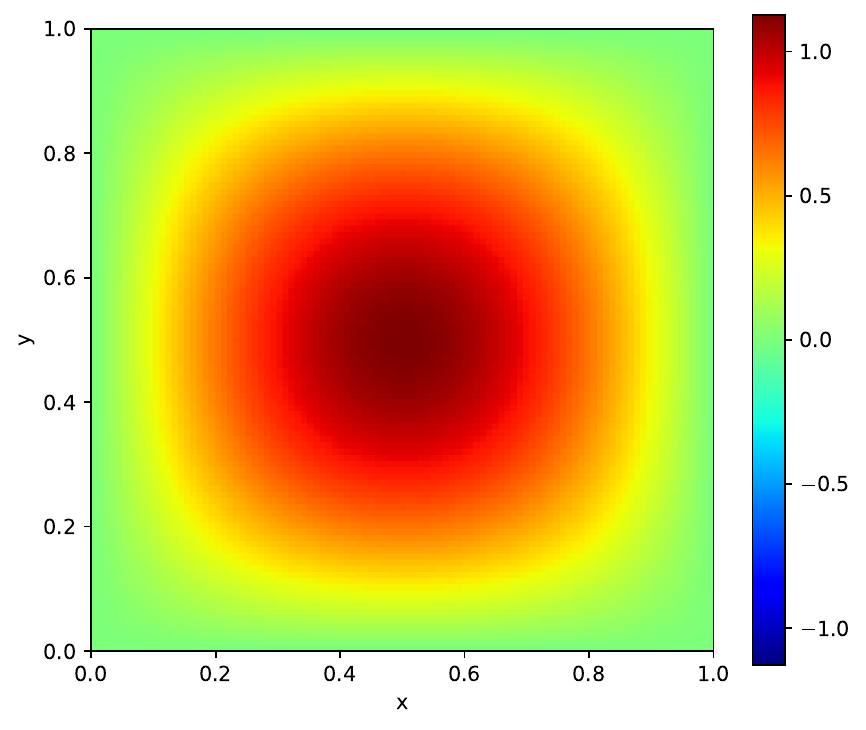}
        \caption{$\Re(u_{T})$}
    \end{subfigure}\hfill
    \begin{subfigure}[b]{0.32\textwidth}
        \centering
        \includegraphics[width=\textwidth]{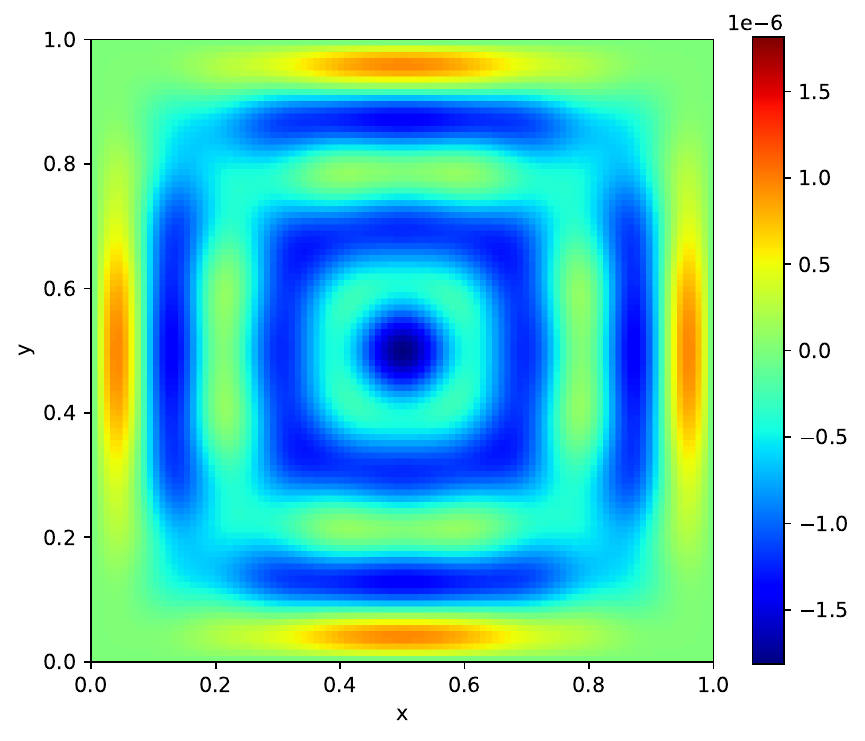}
        \caption{$\Re\big(\Psi(f)-u_{T}\big)$}
    \end{subfigure}
    \caption{Final-time comparison: real part. Parameters \((N_x,N_y,N_t)=(100,100,70)\), \(\tau = 10^{-5}\), \(\varepsilon=10^{-5}\).}
    \label{fig13}
\end{figure}

\begin{figure}[htbp]
    \centering
    \begin{subfigure}[b]{0.32\textwidth}
        \centering
        \includegraphics[width=\textwidth]{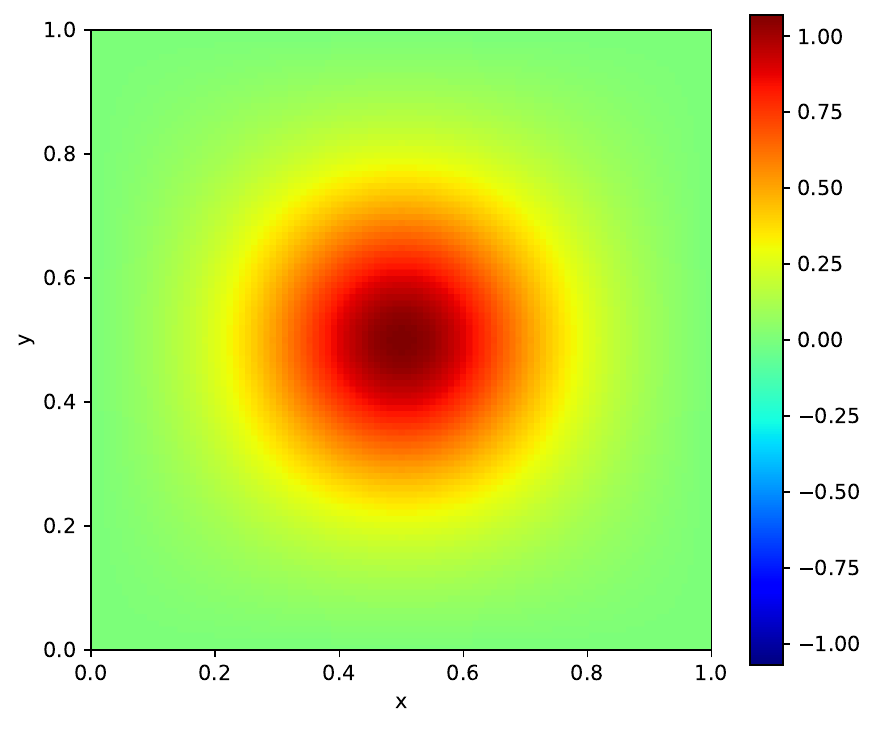}
        \caption{$\Im\big(\Psi(f)\big)$}
    \end{subfigure}\hfill
    \begin{subfigure}[b]{0.32\textwidth}
        \centering
        \includegraphics[width=\textwidth]{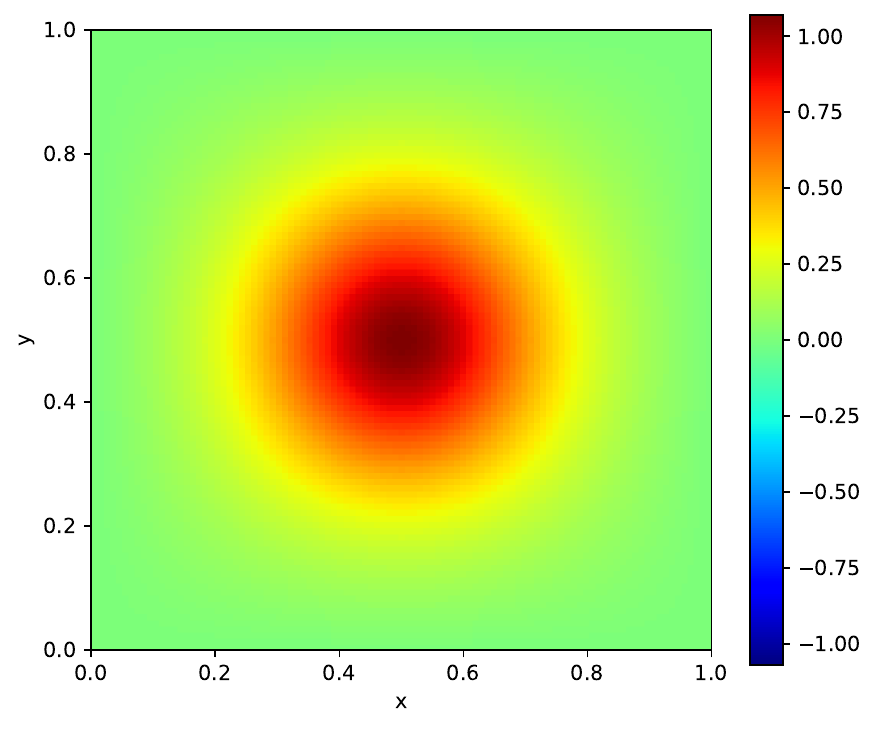}
        \caption{$\Im(u_{T})$}
    \end{subfigure}\hfill
    \begin{subfigure}[b]{0.32\textwidth}
        \centering
        \includegraphics[width=\textwidth]{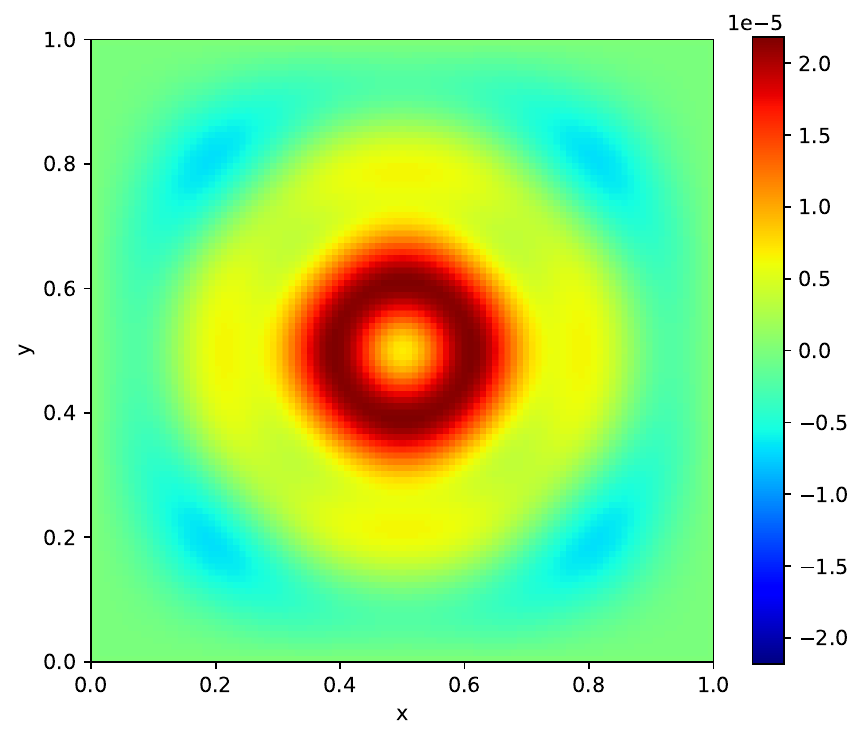}
        \caption{$\Im\big(\Psi(f)-u_{T}\big)$}
    \end{subfigure}
    \caption{Final-time comparison: imaginary part. Parameters \((N_x,N_y,N_t)=(100,100,70)\), \(\tau = 10^{-5}\), \(\varepsilon=10^{-5}\).}
    \label{fig14}
\end{figure}

Figures~\ref{fig15}--\ref{fig16} display the recovered forcing \(f_{\mathrm{rec}}\) and the ground-truth \(f\) from \eqref{eq:ex4f} at the final time slice \(t=T\). The reconstruction of the state is essentially exact, whereas the recovered forcing exhibits a moderate discrepancy. The relative (discrete) space–time error is
\[
\frac{\Vert f_{\mathrm{rec}}-f\Vert_{h,t}^{2}}{\Vert f \Vert_{h,t}^{2}}
\;=\; 3.1720\times 10^{-1}.
\]

\begin{figure}[htbp]
    \centering
    \begin{subfigure}[b]{0.32\textwidth}
        \centering
        \includegraphics[width=\textwidth]{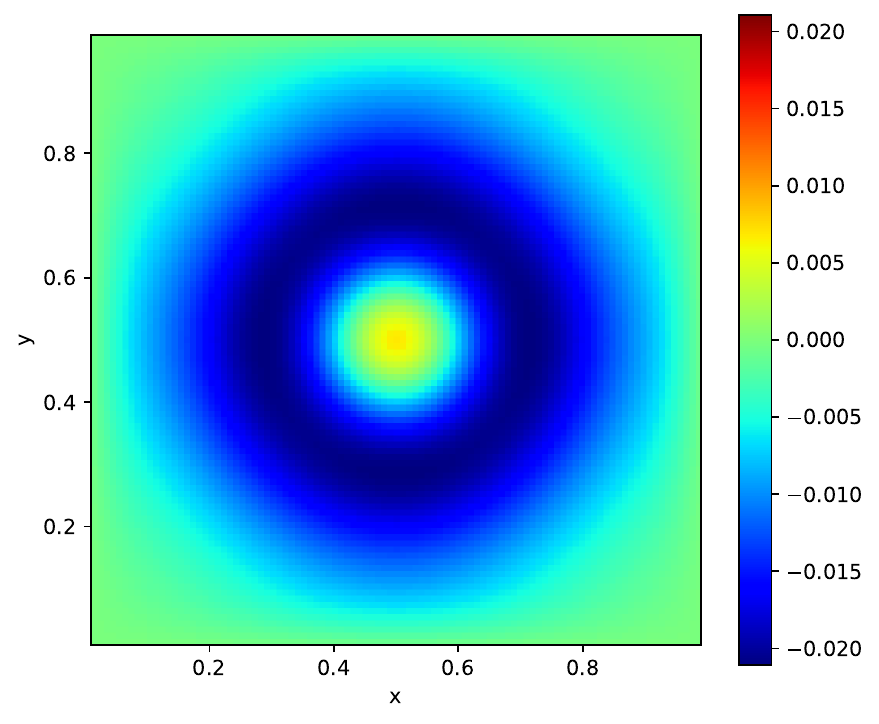}
        \caption{$\Re\big(f_{\mathrm{rec}}(\cdot,T)\big)$}
    \end{subfigure}\hfill
    \begin{subfigure}[b]{0.32\textwidth}
        \centering
        \includegraphics[width=\textwidth]{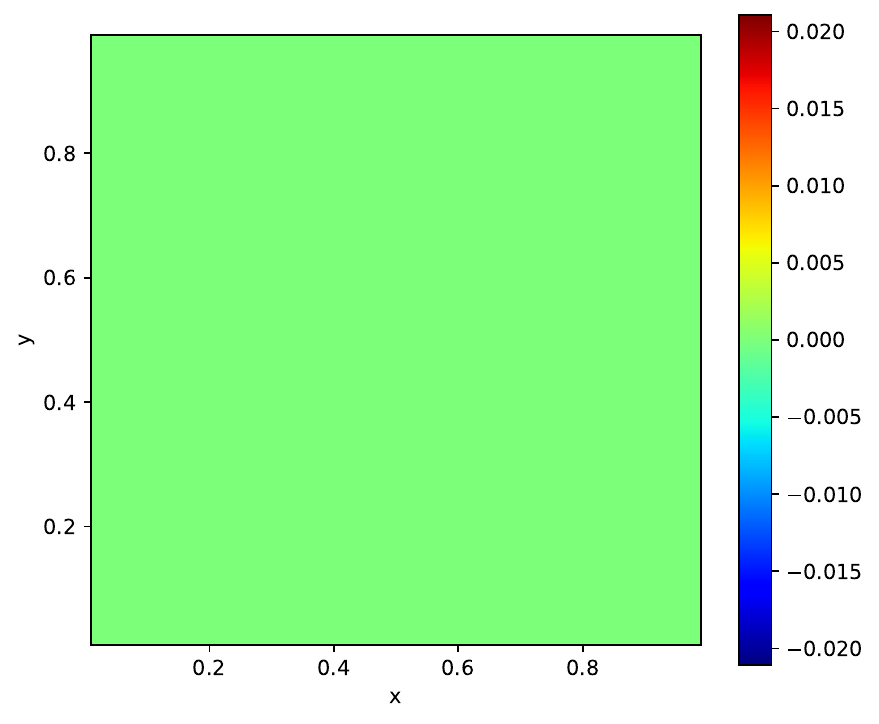}
        \caption{$\Re\big(f(\cdot,T)\big)$}
    \end{subfigure}\hfill
    \begin{subfigure}[b]{0.32\textwidth}
        \centering
        \includegraphics[width=\textwidth]{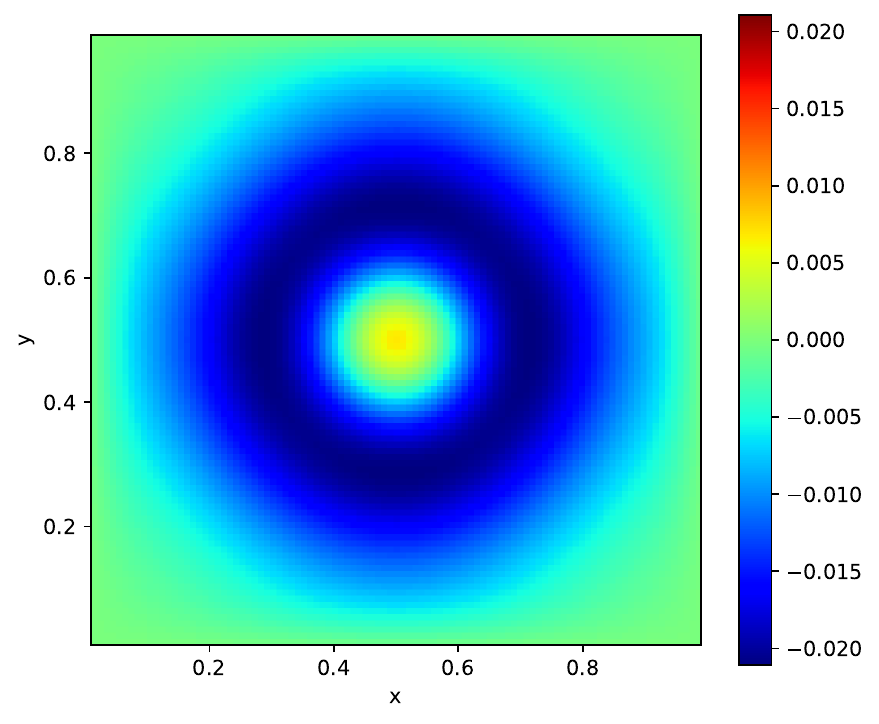}
        \caption{$\Re\big(f_{\mathrm{rec}}(\cdot,T)-f(\cdot,T)\big)$}
    \end{subfigure}
    \caption{Recovered vs.\ true forcing at \(t=T\) (real part). Parameters \((N_x,N_y,N_t)=(100,100,70)\), \(\tau = 10^{-5}\), \(\varepsilon=10^{-5}\).}
    \label{fig15}
\end{figure}

\begin{figure}[htbp]
    \centering
    \begin{subfigure}[b]{0.32\textwidth}
        \centering
        \includegraphics[width=\textwidth]{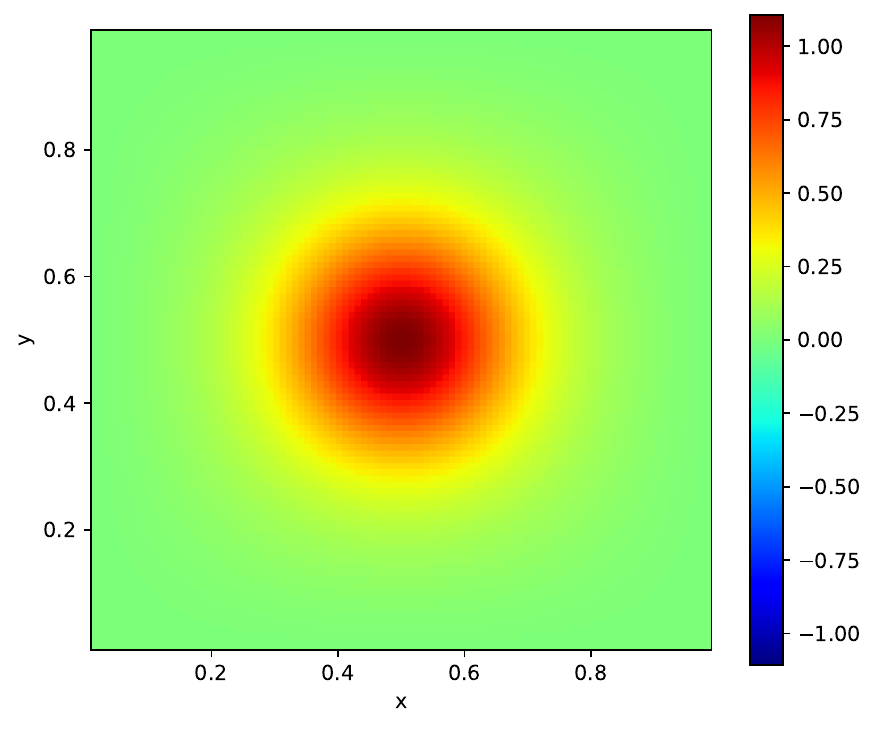}
        \caption{$\Im\big(f_{\mathrm{rec}}(\cdot,T)\big)$}
    \end{subfigure}\hfill
    \begin{subfigure}[b]{0.32\textwidth}
        \centering
        \includegraphics[width=\textwidth]{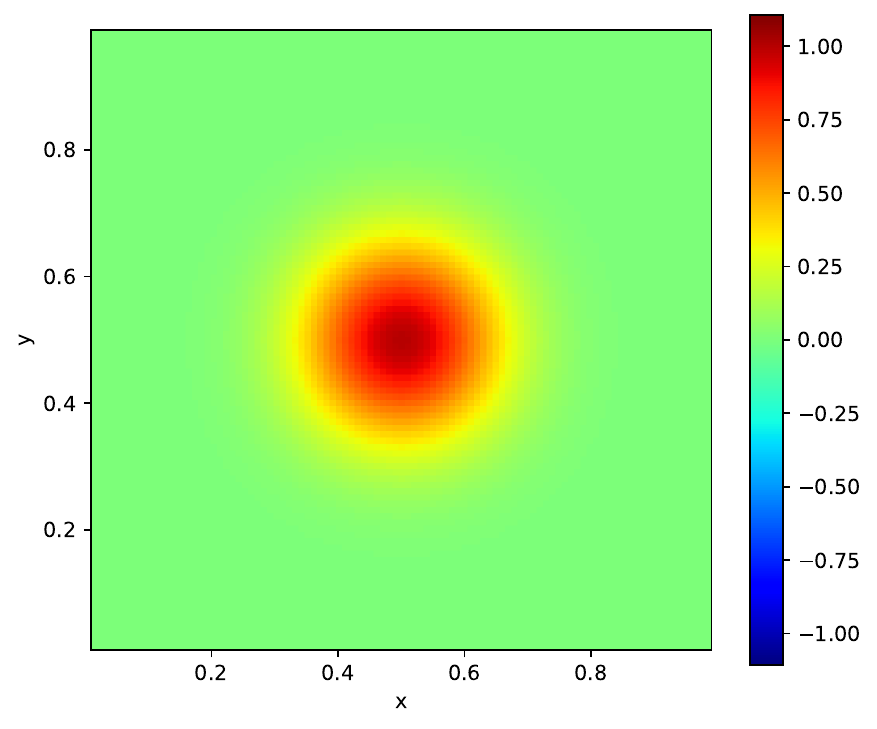}
        \caption{$\Im\big(f(\cdot,T)\big)$}
    \end{subfigure}\hfill
    \begin{subfigure}[b]{0.32\textwidth}
        \centering
        \includegraphics[width=\textwidth]{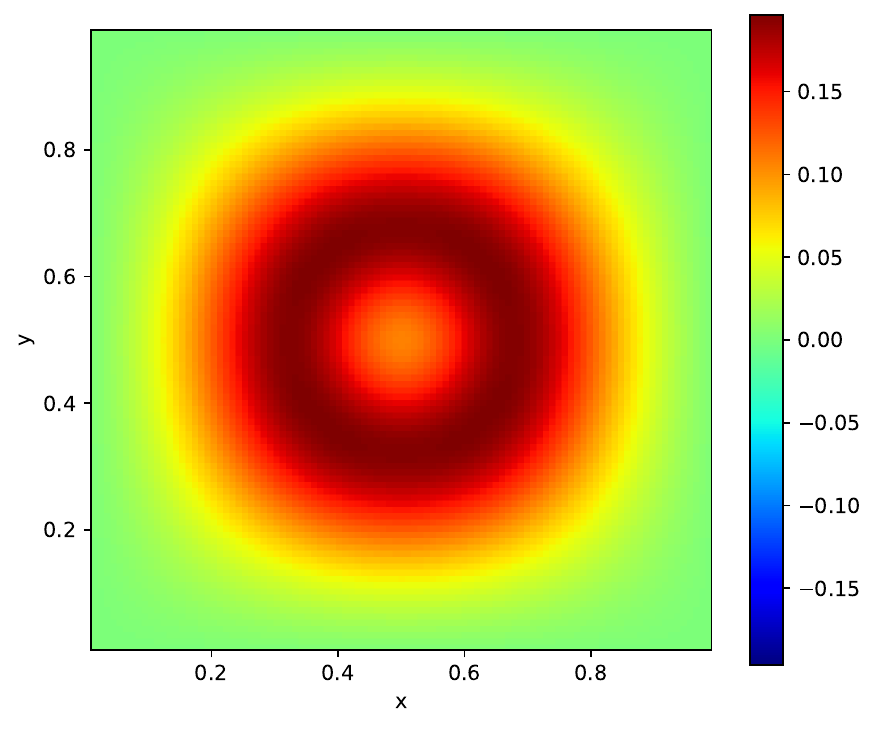}
        \caption{$\Im\big(f_{\mathrm{rec}}(\cdot,T)-f(\cdot,T)\big)$}
    \end{subfigure}
    \caption{Recovered vs.\ true forcing at \(t=T\) (imaginary part). Parameters \((N_x,N_y,N_t)=(100,100,70)\), \(\tau = 10^{-5}\), \(\varepsilon=10^{-5}\).}
    \label{fig16}
\end{figure}

\begin{remark}[Identifiability with a single terminal observation]
In space–time source inversion from a single time slice, many different forcings \(f\) can produce nearly identical terminal states \(u(\cdot,T)\). Consequently, while \(\Psi(f)\) matches \(u_T\) to high accuracy, the recovery of \(f\) is only approximate in the space–time norm. Incorporating additional data (e.g., multi-time observations) or temporal regularization can improve identifiability and reduce the reconstruction error.
\end{remark}.

\section{Conclusions and open problems}
\label{section:Summary:perspectives}

In this article, we have investigated an inverse source problem for the Ginzburg-Landau equation, aiming to recover an unknown space-time dependent forcing term from final-time observations. Within a weak-solution framework, we formulated the reconstruction task as the minimization of a Tikhonov-regularized functional and derived an explicit gradient expression through an adjoint system.

We proved the Lipschitz continuity of this gradient, ensuring both the theoretical stability of the optimization procedure and the convergence of gradient-based numerical schemes. Furthermore, we established existence and uniqueness results for quasi-solutions, thus providing a rigorous variational foundation for the inverse identification process.

From the numerical point of view, we implemented an adjoint-based optimization algorithm relying on Crank-Nicolson time discretization and second-order finite differences in space. The resulting method was tested on several two-dimensional examples, confirming its ability to accurately reconstruct complex-valued sources even under moderate noise perturbations.

The numerical evidence aligns with the analytical predictions: the regularized functional exhibits a well-behaved convex structure, and the reconstructed source converges toward the exact ones as the tolerance and regularization parameters are refined.

Beyond these specific findings, our study highlights the broader significance of the Ginzburg-Landau (GL) equation as a model for dissipative-dispersive dynamics. Because the GL framework appears in contexts as diverse as nonlinear optics, hydrodynamic instabilities, and chemical pattern formation, advances in its inverse analysis have wide interdisciplinary implications. The techniques developed here therefore provide a foundation not only for the reconstruction of hidden excitations in amplitude-equation models but also for related PDE-constrained optimization and control problems.

At the same time, several important directions remain open and deserve future exploration:
\begin{enumerate}
    \item {\bf Nonlinear extensions.} Extending the present analysis to the cubic and cubic-quintic GL equations would bring the inverse-problem framework closer to realistic physical models. The main difficulty lies in handling the nonlinear coupling between amplitude and phase, which may lead to multiple local minima and non-convex optimization landscapes. 

    \item {\bf Quantitative stability estimates.} While the current work ensures Lipschitz continuity of the gradient, obtaining explicit conditional stability bounds for the reconstructed source (possibly of logarithmic or H\"older type) remains an open theoretical challenge.
    \item {\bf Alternative regularization strategies.} Incorporating sparsity-promoting, adaptive, or data-driven regularization (e.g., $L^1$ or total-variation penalties) bounds for the reconstructed source (e.g., possibly of logarithmic or H\"older type) remains an open theoretical challenge.
    \item {\bf Numerical and computational challenges.} Large-scale simulations of the GL equation involve stiff, highly oscillatory dynamics due to the coexistence of diffusion and dispersion. Developing structure-preserving or phase-accurate discretizations (such as exponential integrators or operator-splitting schemes) may significantly enhance the stability and efficiency of inversion algorithms.
    \item {\bf Integration with machine-learning surrogates.} Combining adjoint-based inversion with neural surrogates offers a promising path to accelerate computations and to infer effective source representations in nonlinear or high-dimensional regimes.
\end{enumerate}

In summary, this article provides a rigorous and computationally validated framework for the inverse identification of source terms in the Ginzburg-Landau equation. By unifying analytical well-posedness, adjoint-based gradient characterization, and efficient numerical implementation, it establishes a solid platform for further research on dissipative-dispersive inverse problems.

Addressing the open challenges listed above will deepen the mathematical understanding of these systems and foster new connections between (nonlinear) inverse problems and modern computational approaches.

\section*{Acknowledgments}
R. Morales has received funding from the European Research Council (ERC) under the European Union’s Horizon
2030 research and innovation programme (grant agreement NO: 101096251-CoDeFeL). J. Ramirez was partially supported by Centro de Modelamiento Matemático (CMM) Grant ANID ACE210010 and Basal FB210005.
\appendix

\bibliographystyle{acm}
\bibliography{mybibfile}

\section{Existence and uniqueness results for the linearized GL equation}
\label{section:appendix:existence:uniqueness:results}

This appendix provides standard existence and regularity results for the linearized Ginzburg-Landau equation. These results justify the analytical framework adopted in the main text and ensure the well-posedness of both forward and adjoint systems.

For completeness, we recall the main existence and uniqueness results for the linearized GL equation, which provide the analytical backbone for the weak and strong solution frameworks used throughout this paper. These results, while classical in the parabolic PDE literature, are here adapted to the complex-valued GL operator with mixed dissipative-dispersive terms.

Let $\Omega \subset \mathbb{R}^{N}$ be a bounded domain with Lipschitz boundary $\partial \Omega$ and $T>0$. Consider the problem
\begin{align}
    \label{problem:existence:uniqueness}
    \begin{cases}
        \pt v -\text{div}((\alpha(x)+\beta(x)i)\nabla v) + \vec{\sigma}(x)\cdot \nabla v + p(x) v=h(x,t)&\text{ in }\Omega\times (0,T),\\
        v=0&\text{ on }\partial \Omega\times (0,T),\\
        v(\cdot,0)=v_0&\text{ in }\Omega.
    \end{cases}
\end{align}

\begin{definition}
    We say that
    \begin{itemize} 
    \item $v$ is a weak solution of \eqref{problem:existence:uniqueness} if  $v$ has the following regularity 
    \begin{align}
        \label{weak:formulation:01}
        v\in L^2(0,T;H_0^1(\Omega)),\quad \pt v\in L^2(0,T;H^{-1}(\Omega)),
    \end{align}
    and if for every test function $\varphi \in H_0^1(\Omega)$ and almost every $t\in (0,T)$, we have
    \begin{align}
        \label{weak:formulation:02}
        \dfrac{d}{dt} \langle v(t),\varphi \rangle_{H^{-1},H_0^1} + \int_\Omega (\alpha(x) + \beta(x)i)\nabla v(t)\cdot \nabla \overline{\varphi}\, dx + \int_\Omega p y\overline{\varphi}\,dx =\int_\Omega h(t)\overline{\varphi}\,dx 
    \end{align}
    with initial condition 
    \begin{align}
        \label{weak:formulation:03}
        v(\cdot,0)=v_0.
    \end{align}
    \item $v$ is a strong solution of \eqref{problem:existence:uniqueness} if $v\in L^2(0,T;H_0^1(\Omega) \cap H^2(\Omega))$ and it satisfies 
    \begin{align*}
        \pt v -\text{div}((\alpha(x) + \beta(x)i)\nabla v) + \vec{\sigma}(x)\cdot \nabla v + pv=h\text{ in }L^2(0,T;L^2(\Omega)),
    \end{align*}
    and the initial condition \eqref{weak:formulation:03}. 
    \end{itemize}
\end{definition}

Concerning the existence of weak solutions for the problem \eqref{problem:existence:uniqueness}, we have the following result:

\begin{proposition}
    \label{proposition:existence:weak:solutions}
    Assume that the coefficients $\alpha,\beta$ and $q$ satisfy
    \begin{align}
        \label{prop:existence:statement:01}
        \alpha,\beta\in L^\infty(\Omega;\mathbb{R}),\quad \alpha(x)\geq \alpha_0 >0 \text{ almost everywhere in }\Omega,
    \end{align}
    and 
    \begin{align}
        \label{prop:existence:statement:02}
        \vec{\sigma}\in [L^\infty(\Omega)]^N\text{ and }q\in L^\infty(\Omega)
    \end{align}
    Moreover, let $v_0\in L^2(\Omega)$ and $h\in L^2(0,T; H^{-1}(\Omega))$. Then, there exists a unique weak solution 
    \begin{align}
        \label{prop:existence:statement:03}
        v\in L^2(0,T;H_0^1(\Omega)) \cap C^0([0,T]; L^2(\Omega)),\quad \pt v\in L^2(0,T;H^{-1}(\Omega)),
    \end{align}
    satisfying the variational formulation \eqref{weak:formulation:02}. Moreover, there exists a constant $C>0$ depending on $\alpha_0$, $\|\alpha\|_{L^\infty}$, $\||\vec{\sigma}|\|_{L^\infty}$ and $\|q\|_{L^\infty}$, $\Omega$ and $T$ such that 
    \begin{align}
        \label{prop:existence:statement:04}
        \max_{t\in [0,T]}\|v(t)\|_{L^2(\Omega)}^2 + \int_0^T \|\nabla v(t)\|^2_{L^2(\Omega)}\,dt \leq C \left( \|v_0\|_{L^2(\Omega)}^2 + \|g\|_{L^2(0,T;H^{-1}(\Omega))}^2 \right).
    \end{align}
\end{proposition}

The existence of weak solutions under mild assumptions on the coefficients ensures that the variational framework is robust. In particular, it shows that the adjoint system, central to the gradient formula, is always well-posed.

\begin{proof}
    We divide the proof into several steps:

    \begin{itemize} 
    \item \textit{Step 1: Galerkin basis and finite dimensional problems.} Let $\{w_k\}_{k\geq 1}$ be an orthonormal basis of $L^2(\Omega)$ consisting of functions in $H_0^1(\Omega)$ (for instance we can consider the eigenfunctions of the (complex) Dirichlet Laplacian). For each $m\in \mathbb{N}$, define the finite-dimensional space $V_m:=\text{span}\{w_1,\ldots,w_m\}$ and the $L^2$-orthogonal projection $P_m:L^2 \to V_m$.
    
    We look for $y_m$ of the form 
    \begin{align*}
        v_m(t);=\sum_{k=1}^m c_k^{(m)} (t) w_k,
    \end{align*}
    with coefficients determined by the finite-dimensional system
    \begin{align}
        \label{proof:existence:1:01}
        \begin{split} 
        &(\pt y_m, \varphi)_{L^2(\Omega)} + \int_\Omega (\alpha+\beta i)\nabla v_m(t)\cdot \nabla \overline{\varphi}\,dx + \int_\Omega (\vec{\sigma} \cdot \nabla v_m + p v_m)\overline{\varphi}\,dx\\
        =&\langle g(t),\varphi \rangle_{H^{-1},H_0^1},\quad \forall \varphi \in V_m,
        \end{split}
    \end{align}
    with initial condition $v_m(0)=P_m v_0$. Notice that \eqref{proof:existence:1:01} defines a linear ODE system for the coefficients $\{c_k^{(m)}\}$. Hence, by classical results, there exists a unique maximal $C^1$ solution on $[0,T]$.

    \item \textit{Step 2: A priori uniform energy estimate in $m$.} Now, we choose $\varphi=v_m(t)$ in \eqref{proof:existence:1:01} and take the real part. Moreover, taking into account that 
    \begin{align*}
        \Re ((\alpha+\beta i)\nabla v\cdot \nabla \overline{v})=a|\nabla v|^2,
    \end{align*}
    and
    \begin{align*}
        \Re \int_\Omega (\vec{\sigma}\cdot \nabla v_m(t) \overline{v}_m + p |v_m|^2)\,dx \geq -\| |\vec{\sigma}|\|_{L^\infty} \cdot \||\nabla v_m|\|_{L^2} \cdot \|v_m\|_{L^2} -\|q\|_{L^\infty} \|v_m\|_{L^2}^2, 
    \end{align*}
    a.e. $t\in (0,T)$, and using Young's inequality we have
    \begin{align}
        \label{proof:existence:1:02}
        \begin{split} 
        &\dfrac{1}{2}\dfrac{d}{dt}\|v_m\|_{L^2}^2 +\alpha_0 \|\nabla v_m\|_{L^2}^2 \\
        \leq & \dfrac{\alpha_0}{2} \|\nabla v_m\|_{L^2}^2 + \left( \dfrac{\| |\sigma|\|_{L^\infty}}{\alpha_0} + \|p\|_{L^\infty} \right) \|v_m\|_{L^2}^2 + \dfrac{1}{\alpha_0} \|h\|_{H^{-1}}^2 .
        \end{split}
    \end{align}

    Equivalently, this means that 
    \begin{align}
        \label{proof:existence:1:03}
        \dfrac{d}{dt} \|v_m\|_{L^2}^2 + \alpha_0 \| |\nabla v_m|\|_{L^2}^2  \leq \gamma_1 \|v_m\|_{L^2}^2 + \dfrac{2}{\alpha_0} \|h\|_{H^{-1}}^2,
    \end{align}
    where $\gamma:=2 \left(\dfrac{\| |\vec{\sigma}|\|_{L^\infty}^2}{\alpha_0} + \|p\|_{L^\infty} \right)$. Now, by Gr\"onwall's inequality, we deduce that 
    \begin{align}
        \label{proof:existence:1:05}
        \|v_m(t)\|_{L^2}^2\leq e^{\gamma_1 t} \left( \|v_0\|_{L^2}^2 + \dfrac{2}{a_0} \int_0^t e^{\gamma_1 s} \|h(s)\|_{H^{-1}}^2 \,ds \right),\quad \forall t\in [0,T].
    \end{align}

    In particular, by \eqref{proof:existence:1:05}, the sequence $(v_m)$ is bounded in $L^\infty(0,T;L^2(\Omega))$. Now, integrating in \eqref{proof:existence:1:03} in time and using \eqref{proof:existence:1:05} we obtain the following uniform bound for the sequence $(\nabla v_m)$:
    \begin{align}
        \label{proof:existence:1:06}
        \int_0^T \| \nabla v_m\|_{L^2(\Omega)}^2\,dt \leq C\left( \|v_0\|_{L^2(\Omega)}^2 + \|h\|_{L^2(0,T;H^{-1}(\Omega))}^2 \right),
    \end{align}
    with $C>0$ independent of $m$.

    \item \textit{Step 3: Bound of the time derivative.} Let us define the operator $\mathcal{A}u :=-\text{div}((a+bi)\nabla u)$, which is continuous from $H_0^{1}(\Omega)$ to $H^{-1}(\Omega)$. Then, from the Galerkin equation \eqref{proof:existence:1:01} viewed in $H^{-1}(\Omega)$, we have 
    \begin{align}
        \label{proof:existence:1:07}
        \|\pt v_m\|_{L^2(0,T;H^{-1}(\Omega))} \leq C \left( \|v_m\|_{L^2(0,T;H_0^1(\Omega))} + \|v_m\|_{L^2(0,T;L^2(\Omega))} + \|h\|_{L^2(0,T;H^{-1}(\Omega))} \right).
    \end{align}

    This shows that $(\pt y_m)$ is a bounded sequence in $L^2(0,T;H^{-1}(\Omega))$ uniformly in $m$.

    \item \textit{Step 4: Compactness and passage to the limit.} By the uniform bounds \eqref{proof:existence:1:05} \eqref{proof:existence:1:06} and \eqref{proof:existence:1:07}, we deduce that $(y_m)$ is bounded in $L^2(0,T;H_0^1(\Omega)) \cap H^1(0,T;H^{-1}(\Omega))$. Now, by Aubin-Lions lemma implies (after extracting a subsequence, still denoted $y_m$) that 
    \begin{align}
        \label{proof:existence:1:08}
        v_m \rightharpoonup v\text{ in }L^2(0,T;H_0^1(\Omega)),\quad v_m \rightarrow v\text{ in }L^2(0,T;L^2(\Omega)),\quad \pt v_m \rightharpoonup \pt v\text{ in }L^2(0,T;H^{-1}(\Omega)). 
    \end{align}

    Now, by \eqref{proof:existence:1:08} and the assumptions \eqref{prop:existence:statement:01} and \eqref{prop:existence:statement:02} on the coefficients $\alpha,\beta, q$ and passing to the limit in \eqref{proof:existence:1:01}, we see that $v$ satisfies the variational formulation \eqref{weak:formulation:02} for every $\varphi\in H_0^1(\Omega)$ and almost everywhere on $t\in (0,T)$. The condition \eqref{weak:formulation:03} follows from the fact that $v\in C^0([0,T];L^2(\Omega))$ (by standard Lions theorems) and the fact that $P_m v_0 \to v_0$ in $L^2(\Omega)$. 

    The uniqueness follows from \eqref{weak:formulation:02} and Gronwall's inequality. This completes the proof of Proposition \ref{proposition:existence:weak:solutions}.
    \end{itemize}
\end{proof}

\begin{proposition}
    \label{proposition:existence:strong:solutions}
    Let $\Omega\subset \mathbb{R}^N$ be a bounded domain with $C^2$ boundary and let $T>0$. Assume that the coefficients $\alpha$ and $\beta$ satisfy 
    \begin{align}
        \label{assumption:prop:existence:02:alpha:beta}
        \alpha,\beta\in W^{1,\infty}(\Omega;\mathbb{R}),\quad \alpha(x)\geq \alpha_0 >0\text{ almost everywhere in }\Omega.
    \end{align}
    Moreover, consider \eqref{prop:existence:statement:02}, $v_0\in H_0^1(\Omega)$ and $h\in L^2(0,T;L^2(\Omega))$. Then, there exists a strong solution $v$ of \eqref{problem:existence:uniqueness} satisfying
    \begin{align*}
        v\in H^1(0,T;L^2(\Omega)) \cap C^0([0,T];H_0^1(\Omega)) \cap L^2(0,T; H^2(\Omega) \cap H_0^1(\Omega)).
    \end{align*}

    Moreover, there exists a constant $C>0$ depending only on $\alpha_0$, $\|\alpha\|_{W^{1,\infty}}$, $\||\vec{\sigma}|\|_{L^\infty}$, $\|p\|_\infty$, $\Omega$ and $T$ such that 
    \begin{align*}
        \begin{split} 
        &\|v\|_{L^2(0,T;H^2(\Omega))} + \|v\|_{C^0([0,T];H_0^1(\Omega))} + \|\pt v\|_{L^2(0,T;L^2(\Omega))} \\
        \leq &C\left(\|h\|_{L^2(0,T;L^2(\Omega))} + \|v_0\|_{H_0^1(\Omega)} \right).
        \end{split}
    \end{align*}
\end{proposition}

The higher regularity of strong solutions is particularly relevant for numerical analysis. It guarantees that discretizations based on finite differences or finite elements converge to the exact solution, justifying the use of the numerical schemes implemented in Section \ref{section:Numerical:experiments}.

\begin{proof}
    We argue by the Galerkin method. 
    \begin{itemize}
        \item \textit{Step 1: Galerkin approximations.} Let $(w_k)_{k\in \mathbb{N}}$ be the Dirichlet Laplacian eigenfunctions. For $m\in \mathbb{N}$, we define the finite-dimensional space $V_m:=\text{span}\{w_1,\ldots,w_m\}$ and set 
        \begin{align*}
            v_m(t)=\sum_{k=1}^m c_k^{(m)} w_k,
        \end{align*}
        which solves the following finite-dimensional system
        \begin{align}
            \label{proof:prop:existence:2:01}
            \langle \pt v_m, \varphi \rangle_{H^{-1},H_0^1} + \int_\Omega (\alpha+\beta i)\nabla v_m\cdot \nabla \overline{\varphi}\,dx + \int_\Omega (\vec{\sigma} \cdot \nabla v_m + pv_m )\overline{\varphi}\,dx =\int_\Omega f\overline{\varphi}\, dx,\quad \forall\, \varphi \in V_m,
        \end{align}
        where we know that by Proposition \ref{proposition:existence:weak:solutions} that for each $m\in \mathbb{N}$, 
        \begin{align}
            \label{proof:prop:existence:2:02}
            v_m \in L^\infty(0,T;L^2(\Omega))\cap L^2(0,T;H_0^1(\Omega)).
        \end{align}
        \item \textit{Step 2: Energy estimates.} Taking $\varphi=\pt y_m$ in \eqref{proof:prop:existence:2:01} and taking real part, we have 
        \begin{align}
            \label{proof:prop:existence:2:03}
            \begin{split} 
            &\|\pt v_m\|_{L^2(\Omega)}^2 + \Re \int_\Omega (\alpha+\beta i)\nabla v_m \cdot \nabla \overline{\pt v_m}\,dx + \Re \int_\Omega (\vec{\sigma}\cdot \nabla v_m + pv_m) \overline{\pt v_m}\,dx \\
            =& \Re (f,\pt v_m).
            \end{split} 
        \end{align}

        Now, integrating by parts in time, using the fact that $\alpha$ is time-independent and by Young's inequality, we have for a.e. $t\in (0,T)$:
        \begin{align*}
            &\|\pt v_m (t)\|_{L^2}^2 + \dfrac{1}{2} \|\sqrt{\alpha} |\nabla v_m(t)|\|_{L^2}^2 \\
            \leq & \||\vec{\sigma}|\|_{L^\infty} \||\nabla v_m(t)|\|_{L^2} \|\pt v_m(t)\|_{L^2}\|\pt v_m(t)\|_{L^2} + \|p\|_{L^\infty} \|v_m(t)\|_{L^2} \|\pt v_m(t)\|_{L^2}\\
            &+ \|h(t)\|_{L^2} \|\pt v_m(t)\|_{L^2}\\
            \leq & \dfrac{1}{2} \|\pt v_m(t)\|_{L^2}^2 + \left( \dfrac{3}{2\alpha_0} \| |\vec{\sigma}|\|_{L^\infty}^2 + \dfrac{3}{2\alpha_0 C_p} \|p\|_{L^\infty}^2 \right) \|\sqrt{\alpha} |\nabla v_m(t)|\|_{L^2}^2 + \dfrac{3}{2} \|h(t)\|_{L^2}^2,
        \end{align*}
        where $C_p>0$ is the constant appeared in the Poincar\'e's inequality. This means that for a.e. $t\in (0,T)$ we have 
        \begin{align}
            \label{proof:prop:existence:2:04}
            \| \pt v_m (t)\|_{L^2}^2 + \dfrac{d}{dt} \|\sqrt{\alpha} |\nabla v_m(t)|\|_{L^2}^2 \leq \gamma_2 \| \sqrt{\alpha}|\nabla v_m(t)|\|_{L^2}^2 + 3\|h(t)\|_{L^2}^2,
        \end{align}
        where $\gamma_2:=\dfrac{3}{\alpha_0} \left( \||\sigma|\|_{L^\infty}^2 + C_p^{-1} \|q\|_{L^\infty}^2  \right)$. By Gronwall's inequality, we have 
        \begin{align}
            \label{proof:prop:existence:2:05}
            \|\sqrt{\alpha} |\nabla v_m(t)|\|_{L^2}^2 \leq e^{\gamma_2 t} \left( \|\alpha\|_{L^\infty} \|\nabla v_m(0)\|_{L^2}^2 + 3\int_0^t e^{-\gamma_2 s} \|h(s)\|_{L^2}^2\,ds \right),\quad \text{ a.e. }t\in [0,T].
        \end{align}

        Moreover, integrating in time in \eqref{proof:prop:existence:2:04} and using the estimate \eqref{proof:prop:existence:2:05}, we deduce that 
        \begin{align}
            \label{proof:prop:existence:2:06}
            \|\pt v_m\|_{L^2(0,T;L^2(\Omega))} \leq C \left( \|v_0\|_{H_0^1(\Omega)} + \|f\|_{L^2(0,T;L^2(\Omega))} \right).
        \end{align}


        \item \textit{Step 3: Estimates on $H^2$.} Let us consider the operator $A:H^2(\Omega) \cap H_0^1(\Omega) \to L^2(\Omega)$ defined by $$Au:=-\text{div}((\alpha+\beta i)\nabla u) + \vec{\sigma}\cdot \nabla u +  pu.$$
        
        Under the assumptions on $\alpha$, $\beta$, $\vec{\sigma}$ and $q$, $A$ enjoys $H^2$-elliptic regularity, i.e., if $g\in L^2(\Omega)$ and $w\in H_0^1(\Omega)$ solves $Aw=g$ in $\Omega$, then $w\in H^2(\Omega) \cap H_0^1(\Omega)$ and we have the following estimate:
        \begin{align*}
            \|w\|_{H^2(\Omega)}\leq C \left(\|g\|_{L^2(\Omega)} + \|w\|_{L^2(\Omega)} \right),
        \end{align*}
        where $C$ depends on $\alpha_0, \|a\|_{W^{1,\infty}}, \|b\|_{W^{1,\infty}}$ and $\Omega$. Now, for a.e. $t\in (0,T)$, the Galerkin function $v_m (t) \in V_m$ satisfies, in the strong sense
        \begin{align*}
            A v_m(t)=g_m(t):=h(t)- \pt v_m(t) -\vec{\sigma}\cdot \nabla v_m(t) - pv_m(t).
        \end{align*}

        It is clear that $g_m(t)\in L^2(\Omega)$ for a.e. $t\in (0,T)$ and the map $t\mapsto \|g_m(t)\|_{L^2(\Omega)}$ belongs to $L^2(0,T)$ uniformly in $m$. Now, by elliptic regularity of $A$, we obtain 
        \begin{align}
            \label{proof:prop:existence:2:08}
            \|v_m(t)\|_{H^2(\Omega)}\leq C \left(\|g_m(t)\|_{L^2(\Omega)} + \|v_m\|_{L^2(\Omega)} \right)\quad \text{a.e. in }(0,T).
        \end{align}
        Now, from \eqref{proof:prop:existence:2:08}, we use the estimates \eqref{prop:existence:statement:04} and \eqref{proof:prop:existence:2:06} to get 
        \begin{align}
            \label{proof:prop:existence:2:09}
            \|v_m\|_{L^2(0,T; H^2(\Omega))} \leq C\left( \|v_0\|_{H_0^1(\Omega)}^2 + \|h\|_{L^2(0,T;L^2(\Omega))} \right).
        \end{align}
        \item \textit{Step 4: Last arrangements.} Now, we will analyze the pass to the limit on $m$. From \eqref{proof:prop:existence:2:02}, \eqref{proof:prop:existence:2:05} and \eqref{proof:prop:existence:2:09}, we know that 
        \begin{align*}
            (y_m)\text{ is bounded in }L^2(0,T;H^2(\Omega)\cap H_0^1(\Omega)) \cap H^1(0,T;L^2(\Omega)).
        \end{align*}

        By Aubin-Lions Lemma (extracting subsequence if necessary) we obtain a limit $v$ with the following properties:
        \begin{align*}
            v_m \rightharpoonup v \text{ in }L^2(0,T;H^2(\Omega)),\quad v_m \rightarrow v \text{ in }L^2(0,T;H_0^1(\Omega)),\text{ and }\pt v_m \rightharpoonup \pt v\text{ in }L^2(0,T;L^2(\Omega)).
        \end{align*}

        Passing to the limit in \eqref{proof:prop:existence:2:01} gives that $v$ is the weak solution and satisfies 
        \begin{align*}
            v\in L^2(0,T;H^2(\Omega) \cap H_0^1(\Omega)),\quad \pt v\in L^2(0,T;L^2(\Omega)).
        \end{align*}

        In fact, this proves that $v$ is a strong solution of \eqref{problem:existence:uniqueness}. Finally, thanks to the above estimates, we conclude that 
        \begin{align*}
            \|v\|_{L^2(0,T;H^2(\Omega) \cap H_0^1(\Omega))} + \|\pt v\|_{L^2(0,T;L^2(\Omega))} \leq C\left( \|v_0\|_{H_0^1(\Omega)} + \|h\|_{L^2(0,T;L^2(\Omega))} \right).
        \end{align*}

        This ends the proof of Proposition \ref{proposition:existence:strong:solutions}.
    \end{itemize}
\end{proof}



\end{document}